\documentclass[12pt,reqno]{amsart}
\usepackage{amsfonts}
\usepackage{amssymb}
\usepackage{amsthm}
\usepackage{amsmath}
\theoremstyle{plain}
\newtheorem {Lem}{Lemma}
\newtheorem {The}{Theorem}

\usepackage{color}

\theoremstyle{remark}

\theoremstyle{definition}

\newtheorem {prob}{Problem}

\newcommand\Label[1]{\label{#1}}
\newcommand{\ep}{\epsilon}

\newif\ifcomm
\let\ifcomm\iffalse
\newcommand{\Form}{A,\Lambda}
\newcommand{\Fideal}[1]{#1,\Gamma} 
\newcommand{\FidealJ}[1]{#1,\Delta}

\newcommand{\FUnT}[2]{\FU(2n,t^{#1}A,t^{#1}#2,t^{#1}\Gamma)} 
\newcommand{\FUnTJ}[2]{\FU(2n,t^{#1}A,t^{#1}#2,t^{#1}\Delta)}
\newcommand{\FUnt}[2]{\FU(2n,t^{#1}#2,t^{#1}\Gamma)} 
\newcommand{\FUnmJ}[2]{\FU^{#2}\left(2n,\frac{#1}{t^{m}},\frac{\Delta}{t^{m}}\right)}

\newcommand{\FFUnt}[1]{\FU(2n,t^{#1}A,t^{#1}\Lambda)}
\newcommand{\FUntO}[2]{\FU^1(2n,t^{#1}#2,t^{#1}\Gamma)} 
\newcommand{\FUntOK}[2]{\FU^K(2n,t^{#1}#2,t^{#1}\Gamma)}
\newcommand{\EUnT}[2]{\FU(2n,t^{#1}#2,t^{#1}#2,t^{#1}\Lambda)}
\newcommand{\EUnt}[2]{\FU(2n,t^{#1}#2,t^{#1}\Lambda)}
\newcommand{\EUnm}[2]{\FU^{#2}\left(2n,\frac{#1}{t^{m}},\frac{\Lambda}{t^{m}}\right)}

\newcommand{\FUnTk}[3]{\FU^{#3}(2n,t^{#1}A,t^{#1}#2,t^{#1}\Gamma)} 
\newcommand{\FUntk}[3]{\FU^{#3}(2n,t^{#1}#2,t^{#1}\Gamma)} 
\newcommand{\LF}{\lfloor}
\newcommand{\RF}{\rfloor}
\newcommand{\gm}{\mathfrak m}

\def\Sp{\operatorname{Sp}}

\def\SO{\operatorname{SO}}

\def\GL{\operatorname{GL}}

\def\SU{\operatorname{SU}}
\def\GU{\operatorname{GU}}
\def\EU{\operatorname{EU}}
\def\FU{\operatorname{FU}}
\def\CU{\operatorname{CU}}

\def\Cent{\operatorname{Cent}}
\def\Max{\operatorname{Max}}

\def\jdim{\operatorname{j-dim}}

\def\map{\longrightarrow}
\def\bar{\overline}
\def\epsilon{\varepsilon}
\def\e{\varepsilon}
\def\a{\alpha}
\def\b{\beta}

\def\Ga{\Gamma}

\title[Relative Commutator Calculus]{Relative unitary commutator calculus, \\
and applications}

\author{Roozbeh Hazrat}
\address{Department of Pure Mathematics, Queen's University,
Belfast BT7 1NN, U.K.}
\email{r.hazrat@qub.ac.uk}
\author{Nikolai Vavilov}
\address{Department of Mathematics and Mechanics,
St.-Petersburg State University, St-Petersburg, Russia}
\email{nikolai-vavilov@yandex.ru}
\author{Zuhong Zhang}
\address{Department of  Mathematics, Beijing Institute of Technology, Beijing, China}
\email{zuhong@gmail.com}
\begin{thanks}
{The first author acknowledges the support of EPSRC first grant
scheme EP/D03695X/1. The work of the second author was supported by
NSh-8464.2006.1 and the RFFI projects 08-01-00756, 09-01-00762,
09-01-00784, 09-01-00878, 09-01-91333, 09-01-90304.
The third author acknowledges the support of NSFC grant 10971011. }
\end{thanks}

\begin{document}

\begin{abstract}
This note revisits localisation and patching method in the setting
of generalised unitary groups. Introducing certain subgroups of
relative elementary unitary groups, we develop relative versions of
the conjugation calculus and the commutator calculus in unitary groups,
which are both more general, and substantially easier than the ones
available in the literature. For the general linear group such relative
commutator calculus has been recently developed by the first and the
third authors. As an application we prove the mixed commutator formula,
$$ [\EU(2n,\Fideal{I}),\GU(2n,J,\Delta)]=
[\EU(2n,\Fideal{I}),\EU(2n,J,\Delta)], $$
\noindent
for two form ideals $(\Fideal{I})$ and $(J,\Delta)$ of a form ring
$(\Form)$. This answers two problems posed in a paper by Alexei Stepanov
and the second author.
\end{abstract}

\maketitle




\section{Introduction}

One of the most powerful ideas in the study of classical groups over rings
is localisation. It allows to reduce many important problems over
various classes of rings subject to commutativity conditions, to
similar problems for semi-local rings. Localisation comes in a
number of versions. The two most familiar ones are {\bf localisation
and patching}, proposed by Daniel Quillen and Andrei Suslin \cite{Sus},
and {\bf localisation--completion}, proposed by Anthony Bak \cite{B4}.
\par
Originally, the above papers addressed the case of the general linear
group $\GL(n,A)$. Soon thereafter, Suslin himself, Vyacheslav Kopeiko,
Marat Tulenbaev, Leonid Vaserstein, Li Fuan, Eiichi Abe, and others
proposed working versions of localisation and patching for other
classical groups, such as symplectic and orthogonal ones, as well
as unitary groups, under some additional simplifying assumptions,
see, for example, \cite{kopeiko,V1,Li1,Li2} and further references in
\cite{NV91,BV3,SV,RN}.
\par
In the most general setting of quadratic modules, similar development
took more time. In fact, the first full scale treatment of
locali\-sa\-tion--completion was proposed only in the Bielefeld Thesis
by the first author \cite{RH,RH2}.
Quite remarkably, the first exhaustive treatment of
localisation and patching came only afterwards, in the St.-Petersburg
Thesis by Victor Petrov \cite{P1} -- \cite{petrov3} and was {\it strongly\/}
influenced by \cite{RH,RH2}.
\par
As a matter of fact, both methods rely on a large body of common calculations,
and technical facts, known as {\bf conjugation calculus} and
{\bf commutator calculus}. Oftentimes these calculations are even referred
to as the {\bf yoga of conjugation}, and the {\bf yoga of commutators},
to stress the overwhelming feeling of technical strain and exertion.
In the unitary case, due to the following circumstances,
\par\smallskip
$\bullet$ the presence of long and short roots,
\par\smallskip
$\bullet$ complicated elementary relations, 
\par\smallskip
$\bullet$ non-commutativity,
\par\smallskip
$\bullet$ non-trivial involution,
\par\smallskip
$\bullet$ non-trivial form parameter,
\par\smallskip\noindent
these calculations tend to be especially lengthy, and highly involved.
\par
A specific motivation for the present work was the desire to create
tools to prove {\it relative\/} versions of structure results for unitary
groups. One typical such result in which we were particularly interested,
is description of subnormal subgroups, or, what is almost the same,
description of subgroups of the unitary groups $\GU(2n,A,\Lambda)$,
normalised by a relative elementary subgroup $\EU(2n,I,\Gamma)$, see
\cite{ZZ} -- \cite{ZZ2}.
\par
Another one was generalisation of the mixed commutator formula
$$ [E(n,R,A),\GL(n,R,B)]=[E(n,R,A),E(n,R,B)], $$
\noindent
proved in the setting of general linear groups by Alexei Stepanov and
the second author~\cite{VS8} where here $R$ is a ring and $A$ and $B$ are two sided ideals of $R$.  This formula is a common generalisation
of the standard commutator formulae. At the stable level, these
formulae were first established in the work of Hyman Bass~\cite{Bass2}.
In another decade, Andrei Suslin, Leonid Vaserstein, Zenon Borevich,
and the second author \cite{Sus,Tul,V1,BV85,SV} discovered that for
commutative rings similar formulae hold for all $n\geq 3$. However,
for two relative subgroups such formulae were proven only at the
stable level, by Alec Mason \cite{Mason74} -- \cite{MAS3}.
\par
However, the proof in \cite{VS8} relied on a strong form of decomposition of
unipotents \cite{SV}, and was not likely to directly generalise to other
classical groups. The authors of \cite{VS8} raised the problems of establishing
this formula via localisation method, and to generalise it to the
general setting of quadratic modules \cite[Problem~1 and Problem~2]{VS8}.
\par
In the paper \cite{RZ} the first and the third authors developed relative
versions of conjugation calculus and commutator calculus in the
general linear group $\GL(n,R)$, thus solving \cite[Problem~1]{VS8}. However,
we believe that the importance and applicability of the method itself
far surpass this immediate application.
\par
In the present paper, which is a sequel of \cite{RZ}, we in a similar way
evolve relative {\it unitary\/} conjugation calculus and commutator
calculus, and, in particular, solve \cite[Problem~2]{VS8}.
Actually, the present paper does not depend on the calculations
from \cite{RH} and \cite{RH2}. Instead, here we establish relative versions of
these results from scratch, in a more general setting. The resulting
versions of conjugation calculus and commutator calculus are both
more general, and {\it substantially easier\/} than the ones
available in the literature.
\par
The overall scheme is always that devised by the first author in
\cite{RH,RH2}.
However, we propose several important technical innovations, and
simplifications. Some such simplifications are similar to those
proposed by the first and the second authors in \cite{RN1}. Most
importantly, following \cite{RZ} we introduce certain subgroups of
relative elementary quadratic groups, and prove all results not at the
absolute, but at the relative level. Another important improvement is
that we notice that the case analysis in the proof of Lemmas \ref{Lem:cong}
and \ref{Lem:comm}
which provide the base of induction, can be cut in half.
\par
As an immediate application of our methods we prove the following
mixed commutator formula.
\begin{The}\label{main}
Let $n\ge 3$, $R$ be a commutative ring, $(\Form)$ be a form ring such
that $A$ is a quasi-finite $R$-algebra. Further, let $(I,\Gamma)$ and
$(J,\Delta)$ be two form ideals of a form ring $(\Form)$. Then
\begin{equation}\Label{eq:final}
\big[\EU(2n,I,\Gamma),\GU(2n,J,\Delta)\big]=
\big[\EU(2n,I,\Gamma),\EU(2n,J,\Delta)\big].
\end{equation}
\end{The}
This theorem is a very broad generalisation of many preceding results,
including the following ones --- which, in turn, generalise a lot
previous of results!
\par\smallskip
$\bullet$ Absolute standard unitary commutator formulae, Bak--Vavilov
\cite{BV3}, Theorem 1.1 and Vaserstein--Hong You \cite{VY}.
\par\smallskip
$\bullet$ Relative unitary commutator formula at the stable level,
under some additional stability assumptions, Habdank \cite{Ha1,Ha2}.
\par\smallskip
$\bullet$ Relative commutator formula for the general linear group
$\GL(n,R)$, Stepanov--Vavilov and Hazrat--Zhang \cite{VS8,VS10,RZ}.
This case is obtained, as one sets in our Theorem, $A=R\oplus R^0$.
\par\smallskip
\par\smallskip
Observe, that in the above generality (relative, without stability
conditions) our results are new already for the following familiar cases.
\par\smallskip
$\bullet$ The case of symplectic groups $\Sp(2l,R)$, when the involution is
trivial, and $\Lambda=R$.
\par\smallskip
$\bullet$ The case of split orthogonal groups $\SO(2l,R)$, when the
involution is trivial and $\Lambda=0$.
\par\smallskip
$\bullet$ The case of classical unitary groups $\SU(2l,R)$,
when $\Lambda=\Lambda_{\max}$.
\par\smallskip\noindent
See \cite{HO} \S5.2B for further discussion on the generalised unitary
groups.
\par
Actually, in \S\S8,9 we give {\it another\/} proof of Theorem~\ref{main},
imitating that of \cite{VS10}. Namely, we show, that Theorem~1 can be
deduced from the {\it absolute\/} standard commutator formula by
careful calculation of levels of the above commutator groups, and
some group-theoretic arguments.
\par
Nevertheless, we believe that our localisation proof, based on the
relative conjugation calculus and commutator calculus, we develop in
\S\S\ref{ghg},\ref{ghgh} of the present paper, and especially the
calculations themselves, are of independent value, and will be used
in many further applications.
\par
The paper is organised as follows. In \S\S2--4 we recall basic
notation, and some background facts, used in the sequel.
The next two sections constitute the technical core of the paper.
Namely, in \S5, and in \S6 we develop relative unitary conjugation
calculus, and relative unitary commutator calculus, respectively.
After that we are in a position to give a localisation proof of
Theorem~1 in \S7. In \S8 we calculate the levels of the mixed
commutator subgroups. Using these calculations in \S9 we give another proof
of Theorem~1, deducing it from the {\it absolute\/} standard commutator
formula. There we also obtain slightly more precise results in some special
situations, for instance, when $A$ itself is commutative or when $I$
and $J$ are comaximal, $I+J=A$. Finally, in \S10 we state and briefly
discuss some further related problems.


\section{Form rings and form ideal}


The notion of $\Lambda$-quadratic forms, quadratic modules and generalised
unitary groups over a form ring $(A,\Lambda)$ were introduced by Anthony
Bak in his Thesis, see \cite{B1,B2}. In this section, and the next one, we
{\it very briefly\/} review the most fundamental notation and results
that will be constantly used in the present paper. We refer to
\cite{B2,HO,BV3,RH,RH2,RN} for details, proofs, and further references.


\subsection{}\label{form algebra}
Let $R$ be a commutative ring with $1$, and $A$ be an (not necessarily commutative)
$R$-algebra. An involution, denoted by $\bar{}$, is an anti-morphism of $A$ of
order $2$. Namely, for  $\alpha,\beta\in A$, one has
$\overline{\alpha+\beta}=\bar\alpha+\bar\beta$,
$\overline{\alpha\beta}=\bar\beta\bar\alpha$ and $\bar{\bar\alpha}=\alpha$.
Fix an element $\lambda\in\Cent(A)$ such that $\lambda\bar\lambda=1$. One may
define two additive subgroups of $A$ as follows:
$$ \Lambda_{\min}=\{\alpha-\lambda\alpha\mid a\in A\}, \qquad
\Lambda_{\max}=\{\alpha\in A\mid \alpha=-\lambda\bar\alpha\}. $$
\noindent
A {\em form parameter} $\Lambda$ is an additive subgroup of $A$ such that
\begin{itemize}
\item[(1)] $\Lambda_{\min}\subseteq\Lambda\subseteq\Lambda_{\max}$,
\smallskip
\item[(2)] $\alpha\Lambda\bar\alpha\subseteq\Lambda$ for all $\alpha\in A$.
\end{itemize}
The pair $(A,\Lambda)$ is called a {\em form ring}.


\subsection{}\label{form ideals}
Let $I\unlhd A$ be a two-sided ideal of $A$. We assume $I$ to be
involution invariant, i.~e.~such that $\bar I=I$. Set
$$ \Gamma_{\max}(I)=I\cap \Lambda, \qquad
\Gamma_{\min}(I)=\{\xi-\lambda\bar\xi\mid\xi\in I\}+
\langle\xi\alpha\bar\xi\mid \xi\in I,\alpha\in\Lambda\rangle. $$
\noindent
A {\em relative form parameter} $\Gamma$ in $(\Form)$ of level $I$ is an
additive group of $I$ such that
\begin{itemize}
\item[(1)] $\Gamma_{\min}(I)\subseteq \Gamma \subseteq\Gamma_{\max}(I)$,
\smallskip
\item[(2)] $\alpha\Gamma\bar\alpha\subseteq \Gamma$ for all $\alpha\in A$.
\end{itemize}
The pair $(I,\Gamma)$ is called a {\em form ideal}.
\par
In the level calculations we will use sums and products of form
ideals. Let $(I,\Gamma)$ and $(J,\Delta)$ be two form ideals. Their sum
is artlessly defined as $(I+J,\Gamma+\Delta)$, it is immediate to verify
that this is indeed a form ideal.
\par
Guided by analogy, one is tempted to set
$(I,\Gamma)(J,\Delta)=(IJ,\Gamma\Delta)$. However, it is considerably
harder to correctly define the product of two relative form parameters.
The papers \cite{Ha1,Ha2,RH,RH2} introduce the following definition
$$ \Gamma\Delta=\Gamma_{\min}(IJ)+{}^J\Gamma+{}^I\Delta, $$
\noindent
where
$$ {}^J\Gamma=\big\langle \xi\Gamma\bar\xi\mid \xi\in J\big\rangle,\qquad
{}^I\Delta=\big\langle \xi\Delta\bar\xi\mid \xi\in I\big\rangle. $$
\noindent
One can verify that this is indeed a relative form parameter of level $IJ$ if $IJ=JI$. Otherwise one needs to consider
the symmetrised product
$$ (I,\Gamma)(J,\Delta)+(J,\Delta)(I,\Gamma)=
\big(IJ+JI,\Gamma_{\min}(IJ+JI)+{}^J\Gamma+{}^I\Delta\big). $$


\subsection{}\label{quasi-finite}
A {\em form algebra over a commutative ring $R$} is a form ring $(A,\Lambda)$,
where $A$ is an $R$-algebra and the involution leaves $R$ invariant, i.e.,
$\bar R=R$.
\par\smallskip
$\bullet$ A form algebra $(\Form)$ is called {\it module finite}, if $A$ is
finitely generated as an $R$-module.
\par\smallskip
$\bullet$ A form algebra $(\Form)$ is called {\it quasi-finite},
if there is a direct system of module finite  $R$-subalgebras $A_i$ of
$A$ such that $\varinjlim A_i=A$.
\par\smallskip
However, in general $\Lambda$ is not an $R$-module. This forces us to
replace $R$ by its subring $R_0$, generated by all $\alpha\bar\alpha$
with $\alpha\in R$. Clearly, all elements in $R_0$ are invariant with
respect to the involution, i.~e.\ $\bar r=r$, for $r\in R_0$.
\par
It is immediate, that any form parameter $\Lambda$ is an $R_0$-module.
This simple fact will be used throughout. This is precisely why we have
to localise in multiplicative subsets of $R_0$, rather than in those of
$R$ itself.


\subsection{}\label{localization}
Let $(A,\Lambda)$ be a form algebra over a commutative ring $R$ with $1$,
and let $S$ be a multiplicative subset of $R_0$, (see \S\ref{quasi-finite}).
For any $R_0$-module $M$ one can consider its localisation $S^{-1}M$
and the corresponding localisation homomorphims $F_S:M\map S^{-1}M$.
By definition of the ring $R_0$ both $A$ and $\Lambda$ are $R_0$-modules,
and thus can be localised in $S$.
\par
In the present paper, we mostly use localisation with respect to the
following two types of multiplication systems of $R_0$.
\par\smallskip
$\bullet$ {\it Principal localisation\/}: for any $s\in R_0$ with $\bar s=s$,
the multiplicative system generated by $s$ is defined as
$\langle s\rangle=\{1,s,s^2,\ldots\}$. The localisation of the form algebra
$(\Form)$ with respect to multiplicative system $\langle s\rangle$ is usually
denoted by $(A_s,\Lambda_s)$, where as usual $A_s=\langle s\rangle^{-1}A$ and
$\Lambda_s=\langle s\rangle^{-1}\Lambda$ are the usual principal localisations
of the ring $A$ and the form parameter $\Lambda$.
Notice that, for each $\alpha\in A_s$, there exists an integer $n$ and an
element $a\in A$ such that $\displaystyle\alpha=\frac a{s^n}$, and for
each $\xi\in\Lambda_s$, there exists an integer $m$ and an element
$\zeta\in\Lambda$ such that $\displaystyle\xi=\frac\zeta{s^m}$.
\par\smallskip
$\bullet$ {\it Maximal localisation\/}: consider a maximal ideal $\gm\in\Max(R_0)$
of $R_0$ and the multiplicative closed set $S_{\gm}=R_0\backslash\gm$. We
denote the localisation of the form algebra $(\Form)$ with respect to $S_{\gm}$
by $(A_\gm,\Lambda_\gm)$, where $A_\gm=S_{\gm}^{-1}A$ and
$\Lambda_\gm=S_{\gm}^{-1}\Lambda$ are the usual maximal localisations of the
ring $A$ and the form parameter, respectively.
\par\smallskip
In these cases the corresponding localisation homomorphisms will be
denoted by $F_s$ and by $F_{\gm}$, respectively.
\par
The following fact is verified by a straightforward computation.
\begin{Lem}\Label{Lem:03}
For any\/ $s\in R_0$ and for any\/ $\gm\in\Max(R_0)$ the
pairs\/ $(A_s,\Lambda_s)$ and\/ $(A_\gm,\Lambda_\gm)$ are form rings.
\end{Lem}


\section{Unitary groups}

In the present section we recall basic notation and facts related to
Bak's generalised unitary groups and their elementary subgroups.


\subsection{}\Label{general}
Let, as above, $A$ be an associative ring with 1. For natural $m,n$
we denote by $M(m,n,A)$ the additive group of $m\times n$ matrices
with entries in $A$. In particular $M(m,A)=M(m,m,A)$ is the ring of
matrices of degree $n$ over $A$. For a matrix $x\in M(m,n,A)$ we
denote by $x_{ij}$, $1\le i\le m$, $1\le j\le n$, its entry in the
position $(i,j)$. Let $e$ be the identity matrix and $e_{ij}$,
$1\le i,j\le n$, be a standard matrix unit, i.e.\ the matrix which has
1 in the position $(i,j)$ and zeros elsewhere.
\par
As usual, $\GL(m,A)=M(m,A)^*$ denotes the general linear group
of degree $m$ over $A$. The group $\GL(m,A)$ acts on the free right
$A$-module $V\cong A^{m}$ of rank $m$. Fix a base $e_1,\ldots,e_{m}$
of the module $V$. We may think of elements $v\in V$ as columns with
components in $A$. In particular, $e_i$ is the column whose $i$-th
coordinate is 1, while all other coordinates are zeros.
\par
Actually, in the present paper we are only interested in the case,
when $m=2n$ is even. We usually number the base
as follows: $e_1,\ldots,e_n,e_{-n},\ldots,e_{-1}$. All other
occuring geometric objects will be numbered accordingly. Thus,
we write $v=(v_1,\ldots,v_n,v_{-n},\ldots,v_{-1})^t$, where $v_i\in A$,
for vectors in $V\cong A^{2n}$.
\par
The set of indices will be always ordered accordingly,
$\Omega=\{1,\ldots,n,-n,\ldots,-1\}$. Clearly, $\Omega=\Omega^+\sqcup\Omega^-$,
where $\Omega^+=\{1,\ldots,n\}$ and $\Omega^-=\{-n,\ldots,-1\}$. For an
element $i\in\Omega$ we denote by $\e(i)$ the sign of $\Omega$, i.e.\
$\e(i)=+1$ if $i\in\Omega^+$, and $\e(i)=-1$ if $i\in\Omega^-$.

\subsection{}\Label{unitary} For a form ring $(\Form)$, one considers the
{\it hyperbolic unitary group\/} $\GU(2n,\Form)$, see~\cite[\S2]{BV3}.
This group is defined as follows:
\par
One fixes a symmetry $\lambda\in\Cent(A)$, $\lambda\bar\lambda=1$ and
supplies the module $V=A^{2n}$ with the following $\lambda$-hermitian form
$h:V\times V\map A$,
$$ h(u,v)=\bar u_1v_{-1}+\ldots+\bar u_nv_{-n}+
\lambda\bar u_{-n}v_n+\ldots+\lambda\bar u_{-1}v_1. $$
\noindent
and the following $\Lambda$-quadratic form $q:V\map A/\Lambda$,
$$ q(u)=\bar u_1 u_{-1}+\ldots +\bar u_n u_{-n} \mod\Lambda. $$
\noindent
In fact, both forms are engendered by a sesquilinear form $f$,
$$ f(u,v)=\bar u_1 v_{-1}+\ldots +\bar u_n v_{-n}. $$
\noindent
Now, $h=f+\lambda\bar{f}$, where $\bar f(u,v)=\bar{f(v,u)}$, and
$q(v)=f(u,u)\mod\Lambda$.
\par
By definition, the hyperbolic unitary group $\GU(2n,A,\Lambda)$ consists
of all elements from $\GL(V)\cong\GL(2n,A)$ preserving the $\lambda$-hermitian
form $h$ and the $\Lambda$-quadratic form $q$. In other words, $g\in\GL(2n,A)$
belongs to $\GU(2n,A,\Lambda)$ if and only if
$$ h(gu,gv)=h(u,v)\quad\text{and}\quad q(gu)=q(u),\qquad\text{for all}\quad u,v\in V. $$
\par
When the form parameter is not maximal or minimal, these groups are not
algebraic. However, their internal structure is very similar to that
of the usual classical groups. They are also oftentimes called general
quadratic groups, or classical-like groups.

\subsection{}\Label{elementary1}
{\it Elementary unitary transvections\/} $T_{ij}(\xi)$
correspond to the pairs $i,j\in\Omega$ such that $i\neq j$. They come in
two stocks. Namely, if, moreover, $i\neq-j$, then for any $\xi\in A$ we set
$$ T_{ij}(\xi)=e+\xi e_{ij}-\lambda^{(\e(j)-\e(i))/2}\bar\xi e_{-j,-i}. $$
\noindent
These elements are also often called {\it elementary short root unipotents\/}.
\noindent
On the other side for $j=-i$ and $\a\in\lambda^{-(\e(i)+1)/2}\Lambda$ we set
$$ T_{i,-i}(\a)=e+\a e_{i,-i}. $$
\noindent
These elements are also often called {\it elementary long root elements\/}.
\par
Note that $\bar\Lambda=\bar\lambda\Lambda$. In fact, for any element
$\a\in\Lambda$ one has $\bar\a=-\bar\lambda\a$ and thus $\bar\Lambda$ coincides with
the set of products $\bar\lambda\a$, $\a\in\Lambda$. This means that in the
above definition $\a\in\bar\Lambda$ when $i\in\Omega^+$ and $\a\in\Lambda$
when $i\in\Omega^-$.
\par
Subgroups $X_{ij}=\{T_{ij}(\xi)\mid \xi\in A\}$, where $i\neq\pm j$, are
called {\it short root subgroups\/}. Clearly, $X_{ij}=X_{-j,-i}$.
Similarly, subgroups $X_{i,-i}=\{T_{ij}(\a)\mid
\a\in\lambda^{-(\e(i)+1)/2}\Lambda\}$ are called {\it long root subgroups\/}.
\par
The {\it elementary unitary group\/} $\EU(2n,\Form)$ is generated by
elementary unitary transvections $T_{ij}(\xi)$, $i\neq\pm j$, $\xi\in A$,
and $T_{i,-i}(\a)$, $\a\in\Lambda$, see~\cite[\S3]{BV3}.

\subsection{}\Label{elementary2}
Elementary unitary transvections $T_{ij}(\xi)$ satisfy the following
{\it elementary relations\/}, also known as {\it Steinberg relations\/}.
These relations will be used throughout this paper.
\par\smallskip
(R1) \ $T_{ij}(\xi)=T_{-j,-i}(\lambda^{(\varepsilon(j)-\varepsilon (i))/2}\bar{\xi})$,
\par\smallskip
(R2) \ $T_{ij}(\xi)T_{ij}(\zeta)=T_{ij}(\xi+\zeta)$,
\par\smallskip
(R3) \ $[T_{ij}(\xi),T_{hk}(\zeta)]=1$, where $h\ne j,-i$ and $k\ne i,-j$,
\par\smallskip
(R4) \ $[T_{ij}(\xi),T_{jh}(\zeta)]=
T_{ih}(\xi\zeta)$, where $i,h\ne\pm j$ and $i\ne\pm h$,
\par\smallskip
(R5) \ $[T_{ij}(\xi),T_{j,-i}(\zeta)]=
T_{i,-i}(\xi\zeta-\lambda^{-\varepsilon(i)}\bar{\zeta}\bar{\xi})$, where $i\ne\pm j$,
\par\smallskip
(R6) \ $[T_{i,-i}(\xi),T_{-i,j}(\zeta)]=
T_{ij}(\xi\zeta)T_{-j,j}(-\lambda^{(\ep(j)-\ep(i))/2}\bar\zeta\xi\zeta)$, where $i\ne\pm j$.
\par\smallskip
Relation (R1) coordinates two natural parametrisations of the same short
root subgroup $X_{ij}=X_{-j,-i}$. Relation (R2) expresses additivity of
the natural parametrisations. All other relations are various instances
of the Chevalley commutator formula. Namely, (R3) corresponds to the
case, where the sum of two roots is not a root, whereas (R4), and (R5)
correspond to the case of two short roots, whose sum is a short root,
and a long root, respectively. Finally, (R6) is the Chevalley commutator
formula for the case of a long root and a short root, whose sum is a root.
Observe that any two long roots are either opposite, or orthogonal, so
that their sum is never a root.


\subsection{}\Label{sub:1.4}
Let $G$ be a group. For any $x,y\in G$, ${}^xy=xyx^{-1}$ and $y^x=x^{-1}yx$
denote the left conjugate and the right conjugate of $y$ by $x$,
respectively. As usual, $[x,y]=xyx^{-1}y^{-1}$ denotes the
left-normed commutator of $x$ and $y$. Throughout the present paper
we repeatedly use the following commutator identities:
\par\smallskip
(C1) \ $[x,yz]=[x,y]\cdot {}^y[x,z]$,
\par\smallskip
(C2) \ $[xy,z]={}^x[y,z]\cdot[x,z]$,
\par\smallskip
(C3) \ ${}^x\bigl[[y,x^{-1}]^{-1},z\bigr] =
{}^y\bigl[x,[y^{-1},z]\bigr]\cdot{}^z\bigl[y,[z^{-1},x]\bigr]$,
\par\smallskip
(C4) \ $[x,{}^yz]={}^y[{}^{y^{-1}}x,z]$,
\par\smallskip
(C5) \ $[{}^yx,z]={}^{y}[x,{}^{y^{-1}}z]$.
\par\smallskip
\noindent
Especially important is (C3), the celebrated {\it Hall--Witt identity\/}.
Sometimes it is used in the following form, known as the {\it three
subgroup lemma\/}.
\begin{Lem}{\label{HW1}}
Let\/ $F,H,L\trianglelefteq G$ be three normal subgroups
of\/ $G$. Then
$$ \big[[F,H],L\big]\le \big [[F,L],H\big ]\cdot \big [F,[H,L]\big ]. $$
\end{Lem}


\section{Relative subgroups}
\label{rel}

In this section we recall definitions and basic facts concerning relative
subgroups.

\subsection{}\Label{relative} One associates with a form ideal $(I,\Gamma)$
the following four relative subgroups.
\par\smallskip
$\bullet$ The subgroup $\FU(2n,I,\Gamma)$ generated by elementary unitary
transvections of level $(I,\Gamma)$,
$$ \FU(2n,I,\Ga)=\big\langle T_{ij}(\xi)\mid \
\xi\in I\text{ if }i\neq\pm j\text{ and }
\xi\in\lambda^{-(\epsilon(i)+1)/2}\Gamma\text{ if }i=-j\big\rangle. $$
\par\smallskip
$\bullet$ The {\it relative elementary subgroup\/} $\EU(2n,I,\Gamma)$
of level $(I,\Gamma)$, defined as the normal closure of $\FU(2n,I,\Gamma)$
in $\EU(2n,\Form)$,
$$ \EU(2n,I,\Ga)={\FU(2n,I,\Ga)}^{\EU(2n,\Form)}. $$
\par\smallskip
$\bullet$ The {\it principal congruence subgroup\/} $\GU(2n,I,\Ga)$ of level
$(I,\Ga)$ in $\GU(2n,A,\Lambda)$ consists of those $g\in \GU(2n,A,\Lambda)$,
which are congruent to $e$ modulo $I$ and preserve $f(u,u)$ modulo $\Ga$,
$$ f(gu,gu)\in f(u,u)+\Ga, \qquad u\in V. $$
\par\smallskip
$\bullet$ The full congruence subgroup $\CU(2n,I,\Gamma)$ of level
$(I,\Gamma)$, defined as
$$ \CU(2n,I,\Ga)=\left \{ g\in \GU(2n,A,\Lambda)\mid
[g,\GU(2n,A,\Lambda)]\subseteq \GU(2n,I,\Ga)\right\}. $$
\par\smallskip
In some books, including \cite{HO}, the group $\CU(2n,I,\Ga)$
is defined differently. However, in many important situations
these definitions yield the same group. Starting from Lemma~6,
this is certainly the case for rings considered in the present
paper.

\subsection{}\Label{relativefacts}
Let us collect several basic facts, concerning relative groups,
which will be used in the sequel. The first one of them asserts that
the relative elementary groups are $\EU(2n,A,\Lambda)$-perfect.
\begin{Lem}\label{hww3}
Suppose either $n\ge 3$ or $n=2$ and $I=\Lambda I+I\Lambda$.
Then
$$ \EU(2n,I,\Gamma)=[\EU(2n,I,\Gamma),\EU(2n,A,\Lambda)]. $$
\end{Lem}
The next lemma gives generators of the relative elementary subgroup
$\EU(2n,I,\Ga)$ as a subgroup. With this end, consider matrices
$$ Z_{ij}(\xi,\zeta)={}^{T_{ji}(\zeta)}T_{ij}(\xi)
=T_{ji}(\zeta)T_{ij}(\xi)T_{ji}(-\zeta), $$
\noindent
where $\xi\in I$, $\zeta\in A$, if $i\neq\pm j$, and
$\xi\in\lambda^{-(\e(i)+1)/2}\Gamma$,
$\zeta\in\lambda^{-(\e(i)+1)/2}\Lambda$, if $i=-j$.
The following result is \cite{BV3}, Proposition 5.1.
\begin{Lem}\label{genelm}
Suppose $n\ge 3$. Then
\begin{multline*}
\EU(2n,I,\Ga)=\big\langle Z_{ij}(\xi,\zeta)\mid \
\xi\in I,\zeta\in\Lambda\text{ if }i\neq\pm j\text{ and }\\
\xi\in\lambda^{-(\epsilon(i)+1)/2}\Gamma,
\zeta\in\lambda^{-(\epsilon(i)+1)/2}\Lambda,
\text{ if }i=-j\big\rangle.
\end{multline*}
\end{Lem}
The following lemma was first established in~\cite{B1}, but remained
unpublished. See~\cite{HO} and~\cite{BV3}, Lemma 4.4, for published
proofs.
\begin{Lem}
The groups $\GU(2n,I,\Gamma)$ and $\CU(2n,I,\Gamma)$ are normal in
$\GU(2n,A,\Lambda)$.
\end{Lem}
The following lemma is the main result of \cite{BV2,BV3}. 
It is usually referred as the {\it absolute standard commutator formula\/}.
Its role in the present paper is two-fold. On the one hand, here we
develop a new and more powerful relative version of the conjugation calculus
and the commutator calculus, which allow, among other things, to give
a new proof of this result. In other words, the localisation proof
of Theorem~\ref{main} proceeds directly in the relative case, and does not depend
on the absolute case. On the other hand, in \S\S8,9 we show that using
level calculations one can deduce Theorem~\ref{main} directly from the absolute case.
\begin{Lem}\label{keylem}
Let $(A,\Lambda)$ be a quasi-finite form ring and
$n\ge 3$. Then for any form ideal $(I,\Gamma)$ the corresponding
elementary subgroup $\EU(2n,I,\Gamma)$ is normal in the hyperbolic
unitary group $\GU(2n,A,\Lambda)$, in other words,
$$ \EU(2n,I,\Gamma)=[\GU(2n,A,\Lambda),\EU(2n,I,\Gamma)]. $$
\noindent
Moreover,
$$ \EU(2n,I,\Gamma)=[\EU(2n,A,\Lambda),\CU(2n,I,\Gamma)]. $$
\end{Lem}


\subsection{}\Label{sub:1.3}
The proofs in the present paper critically depend on the fact that
the functors $\GU_{2n}$ and $\EU_{2n}$ commute with direct limits.
This idea is used twice.
\par\smallskip
$\bullet$ Analysis of the quasi-finite case can be reduced to the case,
where $A$ is module finite over $R_0$, whereas $R_0$ itself is Noetherian.
Indeed, if $(\Form)$ is quasi-finite, (see \S\ref{quasi-finite}),
it is a direct limit $\varinjlim\big((A_j)_{R_j},\Lambda_j\big)$ of
an inductive system of form sub-algebras
$\big((A_j)_{R_j},\Lambda_j\big)\subseteq(A_R,\Lambda)$
such that each $A_j$ is module finite over $R_j$, $R_0\subseteq R_j$ and $R_j$
is finitely generated as an $R_0$-module. It follows that $A_j$ is finitely
generated as an $R_0$-module, see \cite[Cor.~3.8]{RH}. This reduction to module
finite algebras will be used in Lemma~\ref{Lem:08} and Theorem~\ref{main}.
\par\smallskip
$\bullet$ Analysis of any localisation can be reduced to the case of principal
localisations. Indeed, let $S$ be a multiplicative system in a commutative
ring $R$. Then $R_s$, $s\in S$, is an inductive system with respect to the
localisation maps $F_t:R_s\to R_{st}$. Thus, for any functor $\mathcal F$
commuting with direct limits one has ${\mathcal F}(S^{-1}R)=\varinjlim{\mathcal F}(R_s)$.
\par\smallskip
The following crucial lemma relies on both of these reductions. In fact,
starting from the next section, we will be mostly working in the principal
localisation $A_t$. However, eventually we shall have to return to the algebra
$A$ itself. In general, localisation homomorphism $F_S$ is not injective,
so we cannot pull elements of $\GU(2n,S^{-1}A,S^{-1}\Lambda)$ back to $\GU(2n,A,\Lambda)$.
However, over a {\it Noetherian} ring, {\it principal\/} localisation homomorphims
$F_t$ are indeed injective on small $t$-adic neighbourhoods of identity!

\begin{Lem}\Label{Lem:03}
Let $R$ be a commutative Noetherian ring and let $A$ be a module finite
$R$-algebra. Then for any $t\in R$ there exists a positive integer $l$
such that restriction
$$ F_t:\GU(2n,t^lA,t^l\Lambda)\to\GU(2n,A_t,\Lambda_t), $$
\noindent
of the localisation map to the principal congruence subgroup of level
$(t^lA,t^l\Lambda)$ is injective.
\end{Lem}
\begin{proof} Follows from the injectivity of the localisation map
$F_t:t^lA\rightarrow A_t$, see \cite[Lem\-ma~4.10]{B4} or \cite[Lemma~5.1]{RN1}.
\end{proof}


\section{Conjugation Calculus}
\label{ghg}

In the present section we develop a relative version of unitary conjugation
calculus. Throughout this section, we assume that $n\ge 3$, that $(A,\Lambda)$
is a form
ring over a commutative ring $R$ with involution, that $R_0$ is the subring
of $R$, generated by $a\bar a$, where $a\in R$, as in \S\ref{quasi-finite},
and, finally, that $(I,\Gamma)$ and $(J,\Delta)$ are two form ideals of
$(\Form)$.
\par
Clearly, for any $t\neq 0\in R_0$ and any given positive integer $l$, the
set $t^lA$ is in fact an ideal of the algebra $A$. Similarly, it is
straightforward to verify that
$t^l\Lambda=\{t^l\alpha\mid \alpha\in\Lambda\}$
is in fact relative form parameter for $t^lA$, and, thus, $(t^lA,t^l\Lambda)$
is a form ideal.
\par
By the same token, any form ideal $(I,\Gamma)$ gives rise to the form
ideal $(t^lI, t^l\Gamma)$. In particular, we have the corresponding
groups $\FU(2n,t^lA,t^l\Lambda)$ and $\FU(2n,I,\Gamma)$.
\par
Starting from Lemma~\ref{Lem:cong} up to Lemma~\ref{Lem:08}, all calculations
actually take place inside the elementary group $\EU(2n,A_t,\Lambda_t)$, for
some $t\in R_0$. Thus, when we write something like $\FU^1(2n,t^lI,t^l\Gamma)$
or $T_{ij}(t^l\alpha)$ what we {\it really\/} mean is
$F_t\Big(\FU^1(2n,t^lI,t^l\Gamma)\Big)$ or
$T_{ij}\big(F_t(t^l\alpha)\big)$, respectively.
\par
The overall intention of what we are doing in this section, and the next
one, is to perfect the art of getting rid of denominators. We consider
conjugates ${}^xy$ or commutators $[x,y]$, where $x$ may be fractional in
$t$, whereas $y$ is at our disposal. We wish to show that for a given $y$
and any $x$ from a {\it very\/} small $t$-adic neighbourhood of 1 the
elements ${}^xy$ and $[x,y]$ still fall in a reasonably small $t$-adic
neighbourhood of 1. Actually, we aim at such neighbourhood, where
$F_t$ is injective, as in Lemma~\ref{Lem:03}.
\par
For the group $\EU(2n,A_t,\Lambda_t)$ itself, such calculations have
been performed before in the Doktorarbeit of the first author
\cite{RH,RH2}, and have been later used by ourselves, Anthony Bak,
Victor Petrov, and others \cite{BRN,P1,petrov2,petrov3,ZZ}.
\par
What we want to do now, is to develop similar techniques inside the
relative group $\EU(2t,I_t,\Gamma_t)$, where $(I,\Gamma)$ is a form ideal
of the form algebra $(A,\Lambda)$. However, a direct imitation of the
existing proof leads to awkward and unwieldy calculations.
\par
Before, one always carried such calculations in the familiar neighbourhoods
of 1, namely in $\FU(2n,t^lI,t^l\Ga)$ or in $\EU(2n,t^lI,t^l\Ga)$.
However, as it turns out, the first one of them is a bit too small, whereas
the second one is a bit too large. A major new technical point of the
present paper, suggested by the method of our paper \cite{RZ}, is that
calculations become much less cumbersome if one works inside the
subgroup
$$ \FU(2n,t^lI,t^l\Ga)\le\FUnT{l}{I}\le \EU(2n,t^lI,t^l\Ga), $$
\noindent
instead.
\par
By definition, it is the normal closure of $\FUnt{l}{I}$ in $\FFUnt{l}$,
$$ \FUnT{l}{I}={}^{\FFUnt{l}}{\FUnt{l}{I}}\unlhd\FFUnt{l}. $$
\noindent
Normality of $\FUnT{l}{I}$ in $\FFUnt{l}$ will be repeatedly used in
the sequel. Notice, that $\EUnT{l}{A}=\EUnt{l}{A}$.
\par
Let us introduce a further piece of notation. For a form ideal $(I,\Gamma)$
and an element $t\in R_0$, the set $\FU^1\Big(2n,\frac{I}{t^m},\frac{\Gamma}{t^m}\Big)$
consists of elementary unitary transvections $T_{ij}(a)$, such that
$a\in\frac{I}{t^m}$ if $i\neq\pm j$ and $a\in\lambda^{\ep(i)+1)/2}\frac{\Gamma}{t^m}$
if $i=-j$. The set $\FU^1(2n,t^mI,t^m\Gamma)$ is defined similarly. By  $\FU^K(2n,t^mI,t^m\Gamma)$, we mean a product of $K$ (or fewer) elements of  $\FU^1(2n,t^mI,t^m\Gamma)$. 
\par
The following result is based on an induction. As everyone knows, a
journey of a thousand miles starts with the first step, which is usually also the
hardest one. In this case it certainly is.

\begin{Lem}\Label{Lem:cong}
For any given $l,m$ there exists a sufficiently large integer $p$ such that
$$ {}^{\FU^1\big(2n,\frac{A}{t^m},\frac{\Lambda}{t^m}\big)}
\FU^1(2n,t^{4p}I,t^{4p}\Gamma)\subseteq \FUnT{l}{I}. $$
\end{Lem}
\begin{proof}
Suppose that
$$ g= {}^{T_{ij}(a/t^m)}T_{hk}(t^{4p}\alpha)\in
{}^{\FU^1\big(2n,\frac{A}{t^m},\frac{\Lambda}{t^m}\big)}
\FU^1\big(2n,t^{4p}I,t^{4p}\Gamma\big). $$
\noindent
The proof is divided into four cases depending on whether the root
elements $T_{ij}(a/t^m)$ and $T_{hk}(t^{4p}\alpha)$ are short or long.
\par\medskip
\noindent Case I: Both $T_{hk}(t^{4p}\alpha)$ and $T_{ij}(a/t^m)$ are
short root elements, in other words $h\ne\pm k$, $i\ne\pm j$,
and, as above, $\alpha\in I$ and $a\in A$.
\par
The proof breaks into four subcases:
\par\smallskip
(1) $i\ne k$ and $j\ne h$;
\par\smallskip
(2) $i=k$ and $j\ne h$;
\par\smallskip
(3) $i\ne k$ and $j=h$;
\par\smallskip
(4) $i=k$ and $j=h$.
\par\smallskip
We shall prove subcases (1) and (2) and leave it to the reader to reduce subcases
(3)--(4) to subcase (1). In subcase (1), we have further four subcases.
\par\medskip
(i) $i\ne-h$ and $j\ne-k$. Then $T_{hk}(t^{4p}\alpha)$ commutes with
$T_{ij}(a/t^m)$ by Identity (R3). Therefore, $\rho=T_{hk}(t^{4p}\alpha)$
and we are done.
\par\medskip
(ii) $i=-h$ and $j\ne-k$. In this subcase,
$g={}^{T_{ij}(a/t^m)}T_{-ik}(t^{4p}\alpha)$.
\par
If $j=k$, then using (R5) we get
\begin{multline*}
g={}^{T_{ij}(a/t^m)}T_{-i,j}(t^{4p}\alpha)=
T_{-i,j}(t^{4p}\alpha)[T_{-i,j}(-t^{4p}\alpha),{T_{i,j}(a/t^m)}]\\
=T_{-i,j}(t^{4p}\alpha) T_{-j,j}(-\lambda^{(\ep(j)-\ep(-i))/2}
\bar\alpha a t^{4p-m}+\lambda^{(\ep(j)-\ep(i))/2 } \bar a \alpha t^{4p-m})\in\\
\FUnT{4p-m}{I}.
\end{multline*}
\par
If $i\ne k$, then using (R4) we get
\begin{multline*}
g={}^{T_{ij}(a/t^m)}T_{-i,k}(t^{4p}\alpha)
=T_{-i,k}(t^{4p}\alpha)[T_{-i,k}(-t^{4p}\alpha),{T_{i,j}(a/t^m)}]\\
=T_{-i,j}(t^{4p}\alpha) T_{-k,j}(-\lambda^{(\ep(j)-\ep(-i))/2}\bar\alpha at^{4p-m})\in\\
\FUnT{4p-m}{I}.
\end{multline*}

\par\medskip
(iii) $i\ne-h$ and $j=-k$. In this subcase,
$$ g={}^{T_{ij}(a/t^m)}T_{h,-j}(t^{4p}\alpha). $$
\noindent
If $i=h$ then using (R5) we get
\begin{multline*}
g={}^{T_{ij}(a/t^m)}T_{i,-j}(t^{4p}\alpha)
=T_{i,-j}(t^{4p}\alpha)[T_{i,-j}(-t^{4p}\alpha),{T_{i,j}(a/t^m)}]\\
=T_{i,-j}(t^{4p}\alpha) T_{i,-i}(-\lambda^{(\ep(j)-\ep(i))/2}
\bar\alpha a t^{4p-m}+\lambda^{(\ep(-j)-\ep(i))/2 } \bar a \alpha t^{4p-m})\in\\
\FUnT{4p-m}{I}.
\end{multline*}
\par
If $i\ne h$ then using (R4) we get
\begin{multline*}
g={}^{T_{ij}(a/t^m)}T_{h,-j}(t^{4p}\alpha)
=T_{h,-j}(t^{4p}\alpha)[T_{h,-j}(-t^{4p}\alpha),{T_{ij}(a/t^m)}]\\
=T_{h,-j}(t^{4p}\alpha)T_{h,-i}(-\lambda^{(\ep(j)-\ep(i))/2}\bar\alpha a t^{4p-m})\in\\
\FUnT{4p-m}{I}.
\end{multline*}

\par\medskip
(iv) $i=-h$ and $j=-k$. In this subcase,
$\displaystyle g={}^{T_{ij}(a/t^m)}T_{-i,-j}(t^{4p}\alpha)$. By (R1),
$$ g={}^{T_{ij}(a/t^m)}T_{ji}(\lambda^{(\ep(-i)-\ep(-j))/2}t^{4p}\alpha). $$
\noindent
To simplify notation, we denote $\lambda^{(\ep(-i)-\ep(-j))/2}\alpha$ by $\alpha$.
\par
Take an index $q\neq\pm i,\pm j$. Then,
\begin{multline*}
g={}^{T_{ij}(a/t^m)}T_{ji}(t^{4p}\alpha)=
{}^{T_{ij}(a/t^m)}[T_{jq}(t^{2p}),T_{qi}(t^{2p}\alpha)]=\\
=[^{T_{ij}(a/t^m)}T_{jq}(t^{2p}),^{T_{i,j}(a/t^m)}T_{qi}(t^{2p}\alpha)]=\\
=\Bigl[T_{iq} (t^{2p-m}a)T_{jq}(t^{2p}),
T_{qi}(t^{2p}\alpha)T_{qj}(-t^{2p-m}\alpha a)\Bigr].
\end{multline*}
\noindent
Denote the first and the second factors on the right hand side by $x$ and
$y$ respectively. Clearly,
$$ y\in\FU(2n,t^{2p-m}I,t^{2p-m}\Gamma)\qquad\text{and}\qquad
x\in\FU(2n,t^{2p-m}A,t^{2p-m}\Lambda), $$
\noindent
and thus
$$ [x,y]\in\FUnT{2p-m}{I}. $$
\noindent
Now, taking any $p\ge(l+m)/2$ we see that $g\in\FUnT{l}{I}$. This
finishes the proof of subcase.
\par\medskip
In subcase (2), we have
$g={}^{T_{ij}(a/t^m)}T_{hi}(t^{4p}\alpha)
={}^{T_{ij}(a/t^m)}T_{-i,-h}(\lambda^{(\ep(i)-\ep(h))/2}t^{4p}\alpha).$
\noindent
It follows by subcase (1)(ii) that $g\in\FUnT{l}{I}$ for some suitable $p$.
\par
Subcases (3) and (4) can be reduced to subcase (1) in a similar fashion.
\par\medskip
\noindent
Case II: $T_{hk}(t^{4p}\alpha)$ is a short root element and
$T_{ij}(a/t^m)$ is a long root element, i.e., $i=-j$, $h\ne\pm k$,
$\alpha\in I$ and  $\displaystyle a/t^m \in \frac{\Lambda}{t^m}$.
This case is handled by dividing into three subcases:
\par\smallskip
(1) $h\ne-i$ and $k\ne i$. By (R3), $T_{hk}(t^{4p}\alpha)$
commutes with $T_{i,-i}(a/t^m)$. Therefore, $g =T_{hk}(t^{4p}\alpha)$
and we are done.
\par\smallskip
(2) $h=-i$ and $k\ne i$. By (R6) we have
\begin{multline*}
g={}^{T_{i,-i}(a/t^m)}T_{-i,k}(t^{4p}\alpha)=
T_{-i,k}(t^{4p}\alpha)[T_{-i,k}(-t^{4p}\alpha),{T_{i,-i}(a/t^m)}]\\
=T_{-i,k}(t^{4p}\alpha)
T_{-k,k}(\lambda^{(\ep(k)-\ep(-i))/2}t^{8p-m}\bar\alpha a\alpha)
T_{i,k}(t^{4p-m}a\alpha)\in\\
\FUnT{4p-m}{I}.
\end{multline*}
\par\smallskip
(3) $h\ne-i$ and $k=i$. Our claim follows from an argument similar
to that used in subcase (2).
\par\medskip
\noindent
Case III: $T_{hk}(t^{4p}\alpha)$ is a long root element and
$T_{ij}(a/t^m)$ is a short root element. Namely, $i\ne\pm j$, $h=-k$,
$\alpha\in\Gamma$ and $\displaystyle a\in A$. This case is treated by
dividing into three subcases:
\par\smallskip
(1) $i\neq-h$ and $j\ne h$. By (R3), $T_{h,-h}(t^{4p}\alpha)$ commutes
with $T_{ij}(a/t^m)$. Therefore, $g=T_{h,-h}(t^{4p}\alpha)$ and we are
done.
\par\smallskip
(2) $i=-h$ and $j\ne h$. By (R6) we have
\begin{multline*}
g={}^{T_{i,j}(a/t^m)}T_{-i,i}(t^{4p}\alpha)
=T_{-i,i}(t^{4p}\alpha)[T_{-i,i}(-t^{4p}\alpha),{T_{i,j}(a/t^m)}]=\\
=T_{-i,i}(t^{4p}\alpha) T_{-i,j}(-t^{4p-m}\alpha a)
T_{-j,j}(\lambda^{(\ep(j)-\ep(i) )/2}t^{4p-2m}\bar a\alpha a)\in\\
\FUnT{4p-2m}{I}.
\end{multline*}
\par\smallskip
(3) $i\ne-h$ and $j=h$. It follows from an argument similar to that
used in subcase (2).
\par\medskip
\noindent
Case IV: Both $T_{hk}(t^{4p}\alpha)$ and $T_{ij}(a/t^m)$ are
long root elements. Namely, $i=-j$, $h=-k$, $\alpha\in\Gamma$ and
$\displaystyle a/t^m\in\frac{\Lambda}{t^m}$. This case is handled by
dividing into further two subcases:
\par\smallskip
(1)  $i\ne-h$. By (R3), $T_{h,-h}(t^{4p}\alpha)$ commutes with
$T_{i,-i}(a/t^m)$. Therefore, $g=T_{h,-h}(t^{4p}\alpha)$ and we are
done.
\par\smallskip
(2) $i=-h$. Pick an $q\ne\pm i$. Without loss of generality, we
may assume that $\ep(q)=\ep(-i)$. Then by (R6) we have
\begin{multline*}
g={}^{T_{i,-i}(a/t^m)}T_{-i,i}(t^{4p}\alpha)
={}^{T_{i,-i}(a/t^m)}\Bigl(T_{q,i}(t^{3p-m}\alpha)
[T_{-q,i}(t^{p}),T_{q,-q}(t^{2p}\alpha)]\Bigr)\\
=\Bigl({}^{T_{i,-i}(a/t^m)}  T_{q,i}(t^{3p-m} \alpha)\Bigr)
\Bigl[ {}^{{T_{i,-i}(a/t^m)} } T_{-q,i}(t^{p}),
{}^{T_{i,-i}(a/t^m)} T_{q,-q}(t^{2p}\alpha)\Bigr].
\end{multline*}
\noindent
Now, ${}^{T_{i,-i}(a/t^m)}T_{q,-q}(t^{2p}\alpha)$ is trivial by (R3).
By Case II, there is a sufficiently large $p$ such that
$$ {}^{T_{i,-i}(a/t^m)}T_{q,i}(t^{3p-m}\alpha)\in\FUnT{l}{I}, $$
\noindent
and
$$ {}^{{T_{i,-i}(a/t^m)}}T_{-q,i}(t^{p})\in\FFUnt{l}. $$
\noindent
By definition, $\FUnT{l}{I}$ is normalized by $\FFUnt{l}$. Hence, there is
a sufficiently large $p$ such that $g\in\FUnT{l}{I}$. This finishes
the proof of Case IV, hence the whole proof.
\end{proof}

The next lemma immediately follows from Lemma~\ref{Lem:cong} by induction.
\begin{Lem}\Label{Lem:cong2}
For any given $m,l$ there exists a sufficiently large $p$ such that
$$ {}^{\FU^1\big(2n,\frac{A}{t^m},\frac{\Lambda}{t^m}\big)}\FUnt{p}{I}\le\FUnT{l}{I}. $$
\end{Lem}

For further applications we need a stronger fact with $\FUnt{p}{I}$
on the left hand side replaced by its normal closure $\FUnT{p}{I}$
in $\EUnt{p}{A}$.
\begin{Lem}\Label{Lem:cong3}
For any given $m,l$ there exists a sufficiently large $p$ such that
$$ {}^{\FU^1\big(2n,\frac{A}{t^m},\frac{\Lambda}{t^m}\big)}\FUnT{p}{I}\le\FUnT{l}{I}. $$
\end{Lem}
\begin{proof}
We have
\begin{multline*}
{}^{\FU^1\big(2n,\frac{A}{t^m},\frac{\Lambda}{t^m}\big)}\FUnT{p}{I}\\
={}^{\FU^1\big(2n,\frac{A}{t^m},\frac{\Lambda}{t^m}\big)}
\Bigl({}^{\EUnt{p}{A}}{\FUnt{p}{I}}\Bigr)\\
\subseteq{}^{{}^{\FU^1\big(2n,\frac{A}{t^m},\frac{\Lambda}{t^m}\big)}{\EUnt{p}{A}}}
\left({}^{\FU^1\big(2n,\frac{A}{t^m},\frac{\Lambda}{t^m}\big)}{\FUnt{p}{I}}\right).
\end{multline*}
\par
By Lemma~\ref{Lem:cong2}, there exists a sufficiently large $p$ such that
the conjugate in the exponent is contained in
$\FU(2n,t^lA,t^lA,t^l\Lambda)=\FU(2n,t^lA,t^l\Lambda)$,
whereas the conjugate in the base is contained in $\FUnT{l}{I}$.
Since the group $\FUnT{l}{I}$ is normalised by $\FU(2n,t^lA,t^l\Lambda)$,
our claim follows.
\end{proof}

The next lemma is a direct consequence of Lemma~\ref{Lem:cong3}.
Observe, that here we start working with {\it two\/} form ideals
$(I,\Gamma)$ and $(J,\Delta)$.
\begin{Lem}\Label{Lem:cong4}
For any give $m,l$ there exists a sufficiently large $p$ such that
\begin{multline*}
{}^{\FU^1\big(2n,\frac{A}{t^m},\frac{\Lambda}{t^m}\big)}
\big[\FUnT{p}{I},\FU(2n,t^pA,t^pJ,t^p\Delta)\big]\\
\subseteq\big[\FUnT{l}{I},\FU(2n,t^lA,t^lJ,t^l\Delta)\big].
\end{multline*}
\end{Lem}
However, in this lemma, denominators occur in the conjugating elements,
not inside the commutators. To prove our main results, we will have to
face denominators {\it inside\/} the commutator. This is done in the
next section.


\section{Commutator Calculus}
\label{ghgh}

In the present section we develop a relative version of unitary commutator
calculus. As above, we always assume that $n\ge 3$, that $(A,\Lambda)$ is a
form ring over a commutative ring $R$ with involution, that $R_0$ is the subring
of $R$, generated by $a\bar a$, where $a\in R$, and, finally, that $(I,\Gamma)$
and $(J,\Delta)$ are two form ideals of
$(\Form)$. As before,
all calculations take place inside the group $\EU(2n,A_t,\Lambda_t)$.
\begin{Lem}\Label{Lem:comm}
Suppose $m,l,K$ are given. For any $t\in R$ there is an integer $p$, independent of $K$, such that
\begin{multline*}
\Big[\FUntOK{4p}{I},\FU^1\Big(2n,\frac{J}{t^m},\frac{\Delta}{t^m}\Big)\Big]\subseteq\\
[\FUnT{l}{I},\FU(2n,t^lA,t^lJ,t^l\Delta)]. \Label{eq:01}
\end{multline*}
\end{Lem}
\begin{proof}
An easy induction, using identity (C2), shows that 
$$\big [\prod_{i=1}^K u_i,x\big]=\prod_{i=1}^K {}^{\prod_{j=1}^{K-i}u_j}[u_{K-i+1},x],$$ where by convention $\prod_{j=1}^0 u_j=1$. This, with the fact that $\FUnT{l}{I}$ and $\FU(2n,t^lA,t^lJ,t^l\Delta)$  are normalized by 
$\FU(2n,t^pA,t^p\Lambda)$, where $p\geq l$, show that it is enough to establish the lemma for $K=1$, namely, 
\begin{multline*}
\Big[\FUntO{4p}{I},\FU^1\Big(2n,\frac{J}{t^m},\frac{\Delta}{t^m}\Big)\Big]\subseteq\\
[\FUnT{l}{I},\FU(2n,t^lA,t^lJ,t^l\Delta)]. 
\end{multline*}
Let
$$ T_{ij}(t^{4p}\alpha)\in\FUntO{4p}{I},\qquad
\displaystyle T_{hk}\Big(\frac{\beta}{t^m}\Big)\in
\FU^1\Big(2n,\frac{J}{t^m},\frac{\Delta}{t^m}\Big), $$
\noindent
and set
$$ g=\Big[T_{ij}(t^{4p}\alpha),T_{hk}\Big(\frac{\beta}{t^m}\Big)\Big]. $$
\noindent
As in Lemma~\ref{Lem:cong}, we divide the proof into four cases according to
whether root elements $T_{ij}(t^{4p}\alpha)$ and $T_{hk}\Big(\frac{\beta}{t^m}\Big)$
are long or short.
\par\medskip
\noindent Case I.  Both $T_{ij}(t^{4p}\alpha)$ and
$\displaystyle T_{hk}\Big(\frac{\beta}{t^m}\Big)$ are short root elements, i.e.,
$i\ne\pm j$, $h\ne\pm k$, $\alpha\in I$ and $\beta\in J$. The proof breaks
further into following four subcases:
\par\smallskip
(1) $i\ne k$ and $j\ne h$;
\par\smallskip
(2) $i=k$ and $j\ne h$;
\par\smallskip
(3) $i\ne k$ and $j=h$;
\par\smallskip
(4) $i=k$ and $j=h$.
\par\smallskip
We shall prove subcases (1) and (2) and leave it to the reader to reduce
subcases (3) and (4) to subcase (1).  In subcase (1), we have further four
subcases:
\par\smallskip
(i) $i\ne-h$ and $j\ne-k$. By Identity (R3),
$T_{ij}(t^{4p}\alpha)$ commutes with $\displaystyle T_{hk}\Big(\frac{\beta}{t^m}\Big)$.
Therefore, $g=1$ and we are done.
\par\smallskip
(ii) $i=-h$ and $j\ne-k$. In this subcase,
$$ g=\Big[T_{ij}(t^{4p}\alpha),T_{-i,k}\Big(\frac{\beta}{t^m}\Big)\Big]. $$
\noindent
If $j=k$, then by (R5) one has
\begin{multline*}
g=\Big[T_{ij}(t^{4p}\alpha),T_{-i,j}\Big(\frac{\beta}{t^m}\Big)\Big]=\\
T_{-j,j}\big(-\lambda^{(\ep(j)-\ep(i))/2}
\bar\alpha\beta t^{4p-m}+\lambda^{(\ep(j)-\ep(-i))/2 } \bar \beta \alpha t^{4p-m}\big)=\\
\big[T_{ij}(t^{2p}\alpha),T_{-i,j}(t^{2p-m}{\beta})\big]\in\\
[\FUnT{2p}{I},\FUnTJ{2p-m}{J}].
\end{multline*}
If $i\ne k$, then by (R4) one has
\begin{multline*}
g=[T_{ij}(t^{4p}\alpha),T_{-i,k}(\frac{\beta}{t^m}) ]=
T_{-j,k}(-\lambda^{(\ep(j)-\ep(i))/2}\bar\alpha\beta t^{4p-m})=\\
[T_{ij}(t^{2p}\alpha),T_{-i,k}({t^{2p-m}}{\beta})]\in \\
[\FUnT{2p}{I},\FUnTJ{2p-m}{J}].
\end{multline*}
\par
(iii) $i\ne-h$ and $j=-k$.\\
It follows from an argument similar to that used in subcase (ii).
\par\smallskip
(iv) $i=-h$ and $j=-k$.\\
In this subcase, $\displaystyle g=\Big[{T_{ij}(t^{4p}\alpha)},
T_{-i,-j}\Big(\frac{\beta}{t^m}\Big)\Big].$  By (R1) one has
$$ g=\Big[{T_{ij}(t^{4p}\alpha)},\
T_{j,i}\Big(\lambda^{(\ep(-i)-\ep(-j))/2}\frac{\beta}{t^m}\Big)\Big]. $$
\par\smallskip
To simplify notation, we denote $\lambda^{(\ep(-i)-\ep(-j) )/2}{\beta}$ by $\beta$.
Let $q\not = \pm i, \pm j$. Then by (C3) we have
\begin{multline*}
g=\Big[T_{ij}(t^{4p}\alpha),T_{ji}\Big(\frac{\beta}{t^m}\Big)\Big]=
\Big[[T_{i,q}(t^{2p}\alpha),T_{q,j}(t^{2p})],T_{ji}\Big(\frac{\beta}{t^m}\Big)\Big]=\\
{}^{T_{i,q}(t^{2p}\alpha)}
{}^{T_{i,q}(-t^{2p}\alpha)}\Big[[T_{i,q}(t^{2p}\alpha),T_{q,j}(t^{2p})],
T_{ji}\Big(\frac{\beta}{t^m}\Big)\Big].
\end{multline*}
Applying Hall--Witt identity, we get
\begin{multline*}
g={}^{T_{iq}(t^{2p}\alpha)}\Bigl({}^{T_{qj}(t^{2p})}
\Big[T_{iq}(-t^{2p}\alpha),\Big[T_{qj}(-t^{2p}),T_{ji}\Big(\frac{\beta}{t^m}\Big)\Big]\Big]
\times \\
{}^{T_{ji}\big(\frac{\beta}{t^m}\big)}\Big[T_{qj}(t^{2p}),
\Big[T_{ji}\Big(-\frac{\beta}{t^m}\Big),T_{iq}(-t^{2p}\alpha)\Big]\Big]\Bigr).
\end{multline*}
By (R4) this expression can be further rewritten as
$$ g={}^{T_{iq}(t^{2p}\alpha)}\Bigl({}^{T_{qj}(t^{2p})}
\bigl[T_{iq}(-t^{2p}\alpha), T_{qi}(-t^{2p-m}\beta)\bigr]
\cdot
{}^{T_{ji}(\frac{\beta}{t^m})}
\bigl[T_{qj}(t^{2p}),T_{jq}(t^{2p-m}\alpha\beta)\big]\Bigr). $$
\noindent
Clearly, for all $p$ such that $2p-m>l$ the first factor in the base
belongs to
$$ [\FUnT{l}{I},\FUnTJ{l}{J}]. $$
\noindent
On the other hand, the second factor equals
\begin{multline*}
y={}^{T_{ji}(\frac{\beta}{t^m})}
\bigl[T_{qj}(t^{2p}),T_{jq}(t^{2p-m}\alpha\beta)\big]=\\
{}^{T_{ji}(\frac{\beta}{t^m})}\bigl[T_{qj}(t^{2p}),
[T_{ji}(t^{\LF\frac{2p-m}{2}\RF}\beta),T_{iq}(t^{2p-m-\LF\frac{2p-m}{2}\RF}\alpha)]\bigr].
\end{multline*}
\noindent
Set
$$ \displaystyle p'=\max\Big(\Big\LF\frac{2p-m}{2}\Big\RF,
2p-m-\Big\LF\frac{2p-m}{2}\Big\RF\Big). $$
\noindent
Normality of $\FUnT{l}{I}$ implies that
\begin{multline}
\bigl[T_{qj}(t^{2p}),
[T_{ji}(t^{\LF\frac{2p-m}{2}\RF}\beta),
T_{iq}(t^{2p-m-\LF\frac{2p-m}{2}\RF}\alpha)]\bigr]\in\\
[\FUnT{p'}{I},\FUnTJ{p'}{J}].
\end{multline}
\noindent
Hence
$$ y\in{}^{T_{ji}\left(\frac{\beta}{t^m}\right)} [\FUnT{p'}{I},\FUnTJ{p'}{J}]. $$
\noindent
Therefore, by Lemma \ref{Lem:cong4}, for any given $l$,
there is a sufficiently large $p'$  such that,
$$ y\in [\FUnT{l}{I},\FUnTJ{l}{J}]. $$
\par
Summarising the above inclusions for the first and the second factors,
we see that for a sufficiently large $p$, one has
$$ g\in [\FUnT{l}{I},\FUnTJ{l}{J}]. $$
\noindent
This finishes the proof of Subcase (1).
\par\smallskip
In Subcase (2), we have
$$ g=\Big[T_{ij}(t^{4p}\alpha),T_{hi}\Big(\frac{\beta}{t^m}\Big)\Big]=
\Big[T_{ij}(t^{4p}\alpha),
T_{-i,-h}\Big(\lambda^{(\ep(i)-\ep(h))/2}\frac{\beta}{t^m}\Big)\Big]. $$
\noindent
By Subcase (1)(ii) it follows that
$$ g\in[\FUnT{l}{I},\FUnTJ{l}{J}] $$
\noindent
for a suitable $p$.
\par\smallskip
\noindent Case II: $T_{ij}(t^{4p}\alpha)$ is a short root element and
$\displaystyle T_{hk}\Big(\frac{\beta}{t^m}\Big)$ is a long root element,
i.e., $i\ne\pm j$, $h=-k$, $\alpha\in I$ and $\displaystyle{\beta}\in{\lambda^{-(\ep(h)+1)/2}\Delta}$.
This case is handled by dividing into three subcases:
\par\smallskip
(1) $i\ne-h$ and $j\ne h$. By (R3), $T_{ij}$ commutes with $T_{hk}$.
Therefore, $g=1$ and we are done.
\par\smallskip
(2) $i=-h$ and $j\ne h$. By (R6) we have
$$ g=\Big[T_{ij}(t^{4p}\alpha),T_{-i,i}\Big(\frac{\beta}{t^m}\Big)\Big]=
{\Big(T_{-i,j}(\beta\alpha t^{4p-m})T_{-j,j}(-\lambda^{(\ep(j)-\ep(-i))/2}
\bar\alpha\beta\alpha t^{8p-m})\Big)}^{-1}. $$
\par
Further, set
$$ M=\Big\LF\frac{8p-m}{3}\Big\RF,\qquad
M'=\Big(8p-m-2\Big\LF\frac{8p-m}{3}\Big\RF\Big) .$$
\noindent
Then by (R6) one has
\begin{multline*}
g^{-1}=T_{-j,i}(\lambda^{(\ep(j)-\ep(i))/2}t^{4p-m}\alpha\beta)
T_{-j,j}(\lambda^{(\ep(j)-\ep(i))/2}t^{8p-m}\alpha\beta\bar\alpha)=\\
=T_{-j,i}(\lambda^{(\ep(j)-\ep(i))/2}t^{4p-m}\alpha\beta)
T_{-j,j}(\lambda^{(\ep(j)-\ep(i))/2}t^M\alpha t^{M'}\beta
\overline{t^M\alpha})=\\
=T_{-j,i}(\lambda^{(\ep(j)-\ep(i))/2}t^{4p-m}\alpha\beta)
T_{-j,i}(-\lambda^{(\ep(j)-\ep(i))/2}t^{M'+M}\alpha\beta)\times\\
\times[T_{-j,-i}(\lambda^{(\ep(j)-\ep(i))/2}t^{M}\alpha),
T_{-i,i}(t^{M'}{\beta})].
\end{multline*}
Picking an $q\ne\pm i,\pm j$, we see that the first factor of the
above expression equals
\begin{multline*}
T_{-j,i}(\lambda^{(\ep(j)-\ep(i))/2}t^{4p-m}\alpha\beta)
T_{-j,i}(-\lambda^{(\ep(j)-\ep(i))/2}t^{M'+M}\alpha\beta)\\
=[T_{-j,q}(\lambda^{(\ep(j)-\ep(i))/2}t^{\LF 4p-m\RF/2}\alpha),
T_{q,i}(\lambda^{(\ep(j)-\ep(i))/2}t^{4p-m-\LF 4p-m\RF/2}\beta) ]\times\\
\times [T_{-j,q}(\lambda^{(\ep(j)-\ep(i))/2}t^{M'}\alpha),
T_{q,i}(\lambda^{(\ep(j)-\ep(i))/2}t^{M}\beta)].
\end{multline*}
Therefore, for any
$$ \displaystyle p\ge\max\Big(\frac{m+l}{4}+1,\frac{3l+m}{8}+1\Big), $$
\noindent
both factors of $g^{-1}$, and thus also $g^{-1}$ and $g$ themselves,
belong to
$$ [\FUnT{l}{I},\FUnTJ{l}{J}]. $$
\par
(3) $i\ne-h$ and $i=k$. It follows from an argument similar to that
used in Subcase (2).
\par

\smallskip
\noindent Case III: $T_{ij}(t^{4p}\alpha)$ is a long root element and
$T_{hk}\Big(\frac{\beta}{t^m}\Big)$ is a short root element. Namely, $i=-j$,
$h\ne\pm k$, $\alpha\in\Gamma$ and $\displaystyle a\in A$. This case
is treated by dividing into three subcases:
\par\smallskip
(1) $i\ne-h$ and $i\ne k$. By (R3), $T_{i,-i}$ commutes with  $T_{hk}$.
Therefore, $g=1$ and we are done.
\par\smallskip
(2) $i=-h$ and $i\ne k$. By (R6) we have
\begin{multline*}
g=[T_{i,-i}(t^{4p}\alpha),{T_{-i,k}(\frac{\beta}{t^m})}]=
T_{i,k}(\alpha\beta t^{4p-m})
T_{-k,j}(-\lambda^{(\ep(k)-\ep(-i))/2}\bar\beta\alpha\beta t^{4p-2m})=\\
=T_{i,k}(\alpha\beta t^{4p-m})
T_{ik}(-t^{3p-m}\alpha\beta)\cdot
[T_{i,-i}(t^{2p}\alpha),T_{-i,k}(t^{p-m}\beta )].\\
\end{multline*}
When $p\ge m+l$, both factors of the above expression belong to
$$ [\FUnT{l}{I},\FUnTJ{l}{J}]. $$
\par
(3) $i\ne-h$ and $j=h$. It follows from an argument similar to that
used in Subcase~(2).
\par\smallskip\noindent
Case IV: Both $T_{ij}(t^{4p}\alpha)$ and  $T_{hk}\Big(\frac{\beta}{t^m}\Big)$
are long root elements. Namely, $i=-j$, $h=-k$, $\alpha\in\Gamma$ and
$\displaystyle\beta\in\Delta$. This case is handled by further subdividing
it into two subcases.
\par\smallskip
(1)  $i\ne-h$. By (R3), two non-opposite long root elements commute,
and thus $g=1$.
\par\smallskip
(2) $i=-h$. Pick an $q\ne\pm i$. Without loss of generality, we may assume
that $\ep(q)=\ep(-i)$. Then by (R6) we have
\begin{multline*}
g=\Big[T_{i,-i}(t^{4p}\alpha),T_{-i,i}\Big(\frac{\beta}{t^m}\Big)\Big]=
[T_{i,-i}(t^{p}t^{2p}\alpha\overline{t^{p}}),
T_{-i,i}\Big(\frac{\beta}{t^m}\Big)\Big]=\\
=\Big[T_{i,-q}(-t^{3p}\alpha)
[T_{iq}(t^p),T_{q,-q}(t^{2p}\alpha)],T_{-i,i}\Big(\frac{\beta}{t^m}\Big)\Big].
\end{multline*}
By (C2) one has
$$ g={}^{T_{i,-q}(-t^{3p}\alpha)}
\Big[[T_{iq}(t^p),T_{q,-q}(t^{2p}\alpha)],
T_{-i,i}\Big(\frac{\beta}{t^m}\Big)\Big]\cdot
\Big[T_{i,-q}(-t^{3p}\alpha),T_{-i,i}\Big(\frac{\beta}{t^m}\Big)\Big]. $$
\noindent
We claim that for a sufficiently large $p$ both factors on the right hand
side belong to
$$ [\FUnT{l}{I},\FUnTJ{l}{J}]. $$
\noindent
For the second factor this follows from Case II. Thus, it remains to show that
$$ \Big[[T_{i,q}(t^p),T_{q,-q}(t^{2p}\alpha)],T_{-i,i}\Big(\frac{\beta}{t^m}\Big)\Big]\in
[\FUnT{l}{I},\FUnTJ{l}{J}]. $$
\noindent
But
\begin{multline*}
\Big[[T_{i,q}(t^p),T_{q,-q}(t^{2p}\alpha)],T_{-i,i}\Big(\frac{\beta}{t^m}\Big)\Big]=\\
={}^{T_{q,-q}(-t^{2p}\alpha)}{}^{T_{q,-q}(t^{2p}\alpha)}
\Big[[T_{i,q}(t^p),T_{q,-q}(t^{2p}\alpha)],T_{-i,i}\Big(\frac{\beta}{t^m}\Big)\Big].
\end{multline*}
\noindent
By Hall--Witt Identity one has
\begin{multline*}
{}^{T_{q,-q}(-t^{2p}\alpha)}\bigg({}^{T_{i,q}(-t^p)}\Big[T_{q,-q}(-t^{2p}\alpha),
\Big[T_{-i,i}\Big(\frac{\beta}{t^m}\Big),T_{iq}(-t^p)\Big]\Big]\times\\
\times {}^{T_{-i,i}\left(-\frac{\beta}{t^m}\right)}\Big[T_{i,q}(t^p),
\Big[T_{q,-q}(-t^{2p}\alpha),T_{-i,i}\Big(-\frac{\beta}{t^m}\Big)\Big]\Big]\bigg).
\end{multline*}
By (R3) this can be further rewritten as
$$ {}^{T_{q,-q}(-t^{2p}\alpha)}
\bigg({}^{T_{i,q}(t^p)}\Big[T_{q,-q}(-t^{2p}\alpha),
\Big[T_{-i,i}\Big(\frac{\beta}{t^m}\Big),T_{i,q}(-t^p)\Big]\Big]\bigg). $$
\noindent
In turn, by (R6) this is equal to
$$ {}^{T_{q,-q}(-t^{2p}\alpha)}
\bigg({}^{T_{i,q}(t^p)}
\big[ T_{q,-q}(-t^{2p}\alpha),
T_{-i,j}(-t^{p-m}\beta)T_{-q,q}(-\lambda t^{2p-m}\beta)\big]\bigg). $$
\noindent
When $p>l+m$, the commutator in the base belongs to
$$ [\FUnT{l}{I},\FUnTJ{l}{J}]. $$
\noindent
Both $\FUnT{l}{I}$ and $\FUnTJ{l}{J}$ are normalised by $\EUnt{l}{A}$. As
$T_{q,-q}(-t^{2p}\alpha)$ and $T_{i,q}(t^p)$ belong to $\EUnt{l'}{A}$, it follows that
Thus, the first factor also belongs to
$$ [\FUnT{l}{I}, \FUnTJ{l}{J}], $$
\noindent
as claimed.
\par
This finishes the proof of Case IV, and thus the whole proof.
\end{proof}

\begin{Lem}\Label{Lem:left}
Suppose $m,l,K$ are given. For any $t\in R$ there is an integer $p$, independent of $K$, such that
\begin{multline*}
\Big[{}^{\FU^K(2n,t^pA,t^p\Lambda)}\FUntO{p}{I},
\FU^1\Big(2n,\frac{J}{t^m},\frac{\Delta}{t^m}\Big)\Big]\subseteq \\
[\FUnT{l}{I},\FUnTJ{l}{J}].
\end{multline*}
\end{Lem}
\begin{proof}
Let $a,b$ and $c$ be arbitrary elements in  $\FU^K(2n,t^pA,t^p\Lambda),\FUntO{p}{I}$
and  $\FU^1(2n,\frac{J}{t^m}, \frac{\Delta}{t^m})$, respectively. Then by (C2)
one has
\begin{equation}\label{eq:abc}
 [{}^ab,c]=\bigl[b[b^{-1},a],c\bigr]=
\Bigl({}^{b}\bigl[[b^{-1},a],c\bigr]\Bigr)[b,c].
\end{equation}
\noindent
By Lemma~\ref{Lem:comm}, we may find a sufficiently large $p$, such that for
the second factor of  Equation~(\ref{eq:abc}),
$$ [b,c]\in[\FUnT{l}{I},\FUnTJ{l}{J}]. $$
\noindent
Applying Hall--Witt identity to the first of the above factors, we get
$$ {}^{b}\bigl[[b^{-1},a],c\bigr]={}^{ba^{-1}}
\Bigl({}^{a}\bigl[[b^{-1},a],c\bigr]\Bigr)\\
={}^{ba^{-1}}\Bigl({}^{b}\bigl[a^{-1},[c,b]\bigr]
\times{}^{c^{-1}}\bigl[b^{-1},[a^{-1},c^{-1}]\bigr]\Bigr). $$
\noindent
By Lemma~\ref{Lem:comm}, there is a sufficiently large $p$, such that
$$ [c,b]\in[\FUnT{l}{I},\FUnTJ{l}{J}]. $$
\noindent
Furthermore, $a\in\FU^K(2n,t^pA,t^p\Lambda)$ implies that
$$ {}^{b}\bigl[a^{-1},[c,b]\bigr]\in[\FUnT{l}{I},\FUnTJ{l}{J}]. $$
\noindent
Again, Lemma~\ref{Lem:comm} implies that for any given $l'$, there is a
sufficiently large $p$ such that
$$ [a^{-1},c^{-1}]\in [\FFUnt{l'},\FUnTJ{l'}{J}]\subseteq\FUnTJ{l'}{J}. $$
It follows immediately that
$$ \bigl[b^{-1},[a^{-1},c^{-1}]\bigr]\in [\FUnT{l'}{I},\FUnTJ{l'}{J}]. $$
\noindent
Therefore,
\begin{eqnarray*}
{}^{c^{-1}}\bigl[b^{-1},[a^{-1},c^{-1}]\bigr]\subseteq
{}^{{\FU^1(2n,\frac{A}{t^m},\frac{\Lambda}{t^m})}}[\FUnT{l'}{I},\FUnTJ{l'}{J}].
\end{eqnarray*}
Then by Lemma~\ref{Lem:cong4}, we may find a sufficiently large $l'$,
such that
\begin{multline*}
{}^{{\FU^1(2n,\frac{A}{t^m},\frac{\Lambda}{t^m})}}
[\FUnT{l'}{I},\FUnTJ{l'}{J}]\subseteq \\
[\FUnT{l}{I},\FUnTJ{l}{J}].
\end{multline*}
Hence we may find a sufficiently large $p$, such that
$$ {}^{c^{-1}}\bigl[b^{-1},[a^{-1},c^{-1}]\bigr]\in[\FUnT{l}{I},\FUnTJ{l}{J}]. $$
This finishes the proof.
\end{proof}
\begin{Lem}\Label{Lem:left2}
Suppose that $m,l$ are given. For any $t\in R$ there is an integer $p$
such that
\begin{multline*}
\Big[\FUnTk{p}{I}{},\FU^1\Big(2n,\frac{J}{t^m},\frac{\Delta}{t^m}\Big)\Big]\subseteq \\
\big[\FUnT{l}{I},\FUnTJ{l}{J}\big].
\end{multline*}
\end{Lem}
\begin{proof}
Since $\FUnTk{p}{I}{}$ is a group generated by elements of the form $${}^{\FU^K(2n,t^pA,t^p\Lambda)}\FUntO{p}{I}$$ for all natural numbers $K$ and since in Lemma~\ref{Lem:left}, $p$ is independent of $K$, 
the lemma follows from Lemma~\ref{Lem:left} and Identity (C2) by an induction.
\end{proof}
\begin{Lem}\Label{Lem:right}
Suppose $m,l$ are given. For any $t\in R$ there is an integer $p$ such that
\begin{multline*}
\Big[\FUnt{p}{I},{}^{\EUnm{A}{1}}\FUnmJ{J}{1}\Big]\subseteq\\
[\FUnT{l}{I},\FUnTJ{l}{J}].
\end{multline*}
\end{Lem}
\begin{proof} Let
$$ a\in \FUnt{p}{I},\qquad b\in\EUnm{A}{1},\qquad c\in \FUnmJ{J}{1}. $$
\noindent
We consider the commutator $[a,{}^bc]={}^b[{}^{b^{-1}}a,c]$. Lemma~\ref{Lem:cong2}
implies that for any $p'$  there is a sufficiently large $p$ such that
$$ {}^{b^{-1}}a\in\FUnT{p'}{I}. $$
\noindent
Therefore,
$$ [{}^{b^{-1}}a,c]\in \Big[\FUnT{p'}{I},\FUnmJ{J}{1}\Big]. $$
\noindent
By Lemma~\ref{Lem:left2}, for any $p''$ there is a sufficiently large $p'$
such that
\begin{multline*}
[{}^{b^{-1}}a,c]\in\Big[\FUnTk{p'}{I}{},\FUnmJ{J}{1}\Big]\subseteq\\
[\FUnT{p''}{I},\FUnTJ{p''}{J}].
\end{multline*}
Hence ${}^b[{}^{b^{-1}}a,c]$ belongs to
$$ {}^{\EUnm{A}{1}}[\FUnT{p''}{I}, \FUnTJ{p''}{J}]. $$
\noindent
Finally, Lemma~\ref{Lem:cong4} implies that there is a sufficient large $p''$
such that
\begin{multline*}
{}^{\EUnm{A}{1}}[\FUnT{p''}{I}, \FUnTJ{p''}{J}]\subseteq\\
[\FUnT{l}{I}, \FUnTJ{l}{J}].
\end{multline*}
This finishes the proof.
\end{proof}

In the following lemma we use the set
$\EU^K\Big(2n,\frac{J}{t^m},\frac{\Delta}{t^m}\Big)$ defined as the
set of products of $K$ or fewer elements of
${}^{\EUnm{A}{1}}\FU^1\Big(2n,\frac{J}{t^m},\frac{\Delta}{t^m}\Big)$.

\begin{Lem}\Label{Lem:-1}
Suppose $m,l,K$ are given. For any $t\in R$ there is an integer $p$ such that
\begin{align*}
\Big[\FUntk{p}{I}{1},\,&
\EU^K\Big(2n,\frac{J}{t^m},\frac{\Delta}{t^m}\Big)\Big]\\
&\subseteq\big[\FUnT{l}{I},\FUnTJ{l}{J}\big].
\end{align*}
\end{Lem}
\begin{proof}
The lemma follows from Lemma~\ref{Lem:right} and Lemma \ref{Lem:cong4} and
identity formulas (C1), (C2)  by an  easy induction.
\end{proof}


\section{Mixed Commutator Formula: localisation proof}

In this section we continue to assume that $n\ge 3$, $R$ is a
commutative ring, $(A,\Lambda)$ is a form ring such that $A$ is
a module-finite $R$-algebra, and, finally, $(I,\Gamma)$ and
$(J,\Delta)$ are two form ideals of $(\Form)$.
\par
So far all calculations were taking place in the elementary group
$\EU(2n,A_t,\Lambda_t)$. Now we start to pull them back to the
group $\GU(2n,A,\Lambda)$. The key ingredient is Lemma~\ref{Lem:03},
which asserts that for a suitable positive integer $l$ restriction
$$ F_t:\GU(2n,t^lA,t^l\Lambda)\rightarrow\GU(2n,A_t,\Lambda_t), $$
\noindent
of the localisation homomorphism $F_t$ to the 
congruence subgroup $\GU(2n,t^lA,t^l\Lambda)$ is injective.
\par
Recall, that the functors $\EU_{2n}$ and $\GU_{2n}$ commute with direct
limits. By \S\ref{sub:1.3}, proofs of the following results are reduced
to the case, where $A$ is finite over $R_0$ and $R_0$ itself is Noetherian.
\begin{Lem}\Label{Lem:08}
Let $\gm\in\Max(R_0)$ be a maximal ideal of $R_0$. For any
$g\in\GU(2n,J,\Delta)$, there exists a  $t\in R_0\backslash \gm$ and
an integer $p$, such that
$$ [e,g]\in\big[\EU(2n,I,\Gamma),\EU(2n,J,\Delta)\big], $$
where $e \in \FUntk{p}{I}{1}$. (Here $p$ depends on the choice of $e$.)
\end{Lem}
\begin{proof}
For any maximal ideal $\gm\in\Max(R_0)$, the form ring
$(A_\gm,\Lambda_\gm)$ contains $(J_\gm,\Delta_\gm)$ as a form ideal.
Consider the localisation homomorphism $F_\gm:A\rightarrow A_\gm$ which induces
homomorphisms on the level of unitary groups,
$$ F_\gm:\GU(2n,A,\Lambda)\rightarrow\GU(2n,A_\gm,\Lambda_\gm), $$
\noindent
and
$$ F_\gm:\GU(2n,J,\Delta)\rightarrow\GU(2n,J_\gm,\Delta_\gm). $$
\par
Therefore, for $g\in\GU(2n,\FidealJ{J})$,
$F_\gm(g)\in\GU(2n,J_\gm,\Delta_\gm)$.
Since $A_\gm$ is module finite over the local ring $R_\gm$, $A_\gm$ is semi-local
\cite[III(2.5), (2.11)]{Bass1}, therefore its stable rank is $1$. It follows by
(see \cite[9.1.4]{HO}) that,
$$ \GU(2n,J_\gm,\Delta_\gm)=\EU(2n,J_\gm,\Delta_\gm)\GU(2,J_\gm,\Delta_\gm). $$
\par
Thus, $F_\gm(g)$ can be decomposed as $F_\gm(g)=\ep h$, where
$\ep\in\EU(2n,J_\gm,\Delta_\gm)$ and
$h\in\GU(2,J_\gm,\Delta_\gm)$ is a
$2\times2$ matrix embedded in $\GU(2n,J_\gm,\Delta_\gm)$ and this
embedding can be arranged modulo $\EU(2n,J_\gm,\Delta_\gm)$.
\par
Now, by (\ref{sub:1.3}), we may reduce the problem to the case $A_t$ with
$t\in R_0\backslash \gm$. Namely, $F_t(g)$ is a product of $\ep$ and $h$,
where $\ep\in\EU(2n,J_t,\Delta_t)$, and $h\in\GL(2,J_t,\Delta_t)$.
\par
Therefore $\ep$ is a product of the elementary matrices, thus one has (see~\cite[Prop. 5.1]{BV3})
$$ \ep\in\EU^K\Big(2n,\frac{J}{t^m},\frac{\Delta}{t^m}\Big). $$
\par
Let $e\in\FUntk{p}{I}{1}$. We choose $h$ such that it commutes with
$F_t(e)$. By Lemma~\ref{Lem:-1}, for any given $l$, there is a
sufficiently large $p$ such that
\begin{equation}\label{hmhm}
[F_t(e),F_t(g)]=[F_t(e),\ep]\in[\FUnT{l}{I},\FUnTJ{l}{J}].
\end{equation}
Since $e\in\EU(2n,t^pI,t^p\Gamma)\le\GU(2n,t^lA,t^l\Lambda)$  and
$\GU(2n,t^lA,t^l\Lambda)$ is normal in $\GU(2n,A,\Lambda)$, it follows
$[e,g]\in\GU(2n,t^lA,t^l\Lambda)$. On the other hand, using~(\ref{hmhm}),
one can find
$$ x\in[\FUnT{l}{I},\FUnTJ{l}{J}] $$
\noindent
in $\EU(2n,A,\Lambda)$ such that $F_t(x)=[F_t(e),F_t(g)]$.
Since for suitable $l$, the restriction of $F_t$ to $\GL_n(t^lA,t^l\Lambda)$
is injective by Lemma~\ref{Lem:03}, it follows $[e,g]=x$ and thus
$$ [e,g]\in[\EU(2n,\Fideal{I}), \EU(2n,\FidealJ{J})].$$
\end{proof}
Now, we are prepared to patch the local data. The following lemma
is a key step in the proof of Theorem~1, after that the proof is
finished by an easy induction.
\begin{Lem}\Label{Lem:09}
One has
\begin{equation}\label{popt}
[\FU^1(2n,I,\Gamma),\GU(2n,J,\Delta)]\subseteq
[\EU(2n,I,\Gamma),\EU(2n,J,\Delta)].
\end{equation}
\end{Lem}
\begin{proof}
Let $T_{hk}(\alpha)\in\FU^1(2n,\Fideal{I})$, and $g\in\GU(2n,\Fideal{J})$.
For any maximal ideal $\gm_i\lhd R_0$, choose a $ t_i \in R_0\backslash\gm_i$
and a positive integer $p_i$ according to Lemma~\ref{Lem:08}.
Since the collection of all $t_i^{p_i}$ is not contained in any maximal ideal,
we may find a finite number of $t_i$ and $x_i\in R_0$ such that
$$ \sum_{i}t_i^{p_i}x_i=1. $$
\noindent
We have,
$$ T_{hk}(\alpha)=T_{hk}\Big(\sum_{i}t_i^{p_i}x_i\cdot\alpha\Big)=
\prod_{i}T_{hk}\big(t_i^{p_i}x_i\alpha\big). $$
\noindent
By Lemma~\ref{Lem:08}, it follows immediately that for each $i$,
\begin{multline}\label{kik}
[T_{hk}(t_i^{p_i}x_i\alpha),g]\in
\big[\EU(2n,t_i^{l}I,t_i^{l}\Gamma),\EU(2n,t_i^{l}J,t_i^{l}\Delta)\big]\le\\
\big[\EU(2n,\Fideal{I}),\EU(2n,J,\Delta)\big].
\end{multline}
A direct computation using (\ref{kik}) and Formula (C2) and the fact that
$\EU(2n,\Fideal{I})$ and $\EU(2n,J,\Delta)$ are normal in $\EU(2n,\Form)$,
shows that
$$ [T_{hk}(\alpha),g]=\Big[\prod_{i}T_{hk}(t_i^{p_i}x_i\alpha),g\Big]\in
\big[\EU(2n,\Fideal{I}),\EU(2n,J,\Delta)\big], $$
\noindent
as claimed
\end{proof}
Now we are in a position to finish the proof of Theorem 1.
\begin{proof}[First proof of Theorem $1$]
Since $\FU(n,I,\Gamma)$ is generated by $\FU^1(2n,I,\Gamma)$
whereas $\EU(2n,I,\Gamma)$ and $\EU(2n,J,\Delta)$ are normalised
by $\EU(2n,\Form)$, repeated use of~(\ref{popt}) along with Formula
(C2), gives inclusion
$$ [\FU(2n,I,\Gamma),\GU(2n,J,\Delta)]\le
[\EU(2n,I,\Gamma),\EU(2n,J,\Delta)]. $$
\par
Since $\EU(2n,I,\Gamma)$ is the normal closure of $\FU(2n,I,\Gamma)$
in $\EU(2n,\Form)$, while both $\GU(2n,J,\Delta)$ and
the right-hand side of the above formula are normalised
by $\EU(2n,\Form)$, we get the inclusion
$$ [\EU(2n,I,\Gamma),\GU(2n,J,\Delta)]\le
[\EU(2n,I,\Gamma),\EU(2n,J,\Delta)]. $$
\noindent
The opposite inclusion is obvious.
\end{proof}


\section{Level of the mixed commutators}
\label{level}

In this section we calculate lower and upper levels of
mixed commutators
$$ [\EU(2n,I,\Gamma),\EU(2n,J,\Delta)]. $$
\begin{Lem}
Let\/ $n\ge 2$. Then for any two form ideals  $(I,\Gamma)$ and\/ $(J,\Delta)$
of the form ring\/ $(A,\Lambda)$ one has
$$ \EU(2n,I,\Gamma)\EU(2n,J,\Delta)=\EU(2n,I+J,\Gamma+\Delta). $$
\end{Lem}
\begin{proof}
Additivity of the elementary unitary transvections
$T_{ij}(\a+\b)=T_{ij}(\a)T_{ij}(\b)$, where $i,j \in \Omega$ and $i \not = j$,
while $\a\in I$, $\b\in J$ for $i\neq-j$ and
$\a\in\Gamma$, $\b\in\Delta$ for $i=-j$, implies that the left
hand side contains {\it generators\/} of the right hand side.
The product of two normal subgroups is normal in $\EU(2n,A,\Lambda)$.
\end{proof}
As a preparation to the calculation of lower level, let us observe
that together with \cite[Theorem 2.3]{RH} this lemma implies the
following corollary. Observe that in its turn the proof of
\cite[Theorem 2.3]{RH} heavily depends on Lemma~\ref{genelm}.
\begin{Lem}\label{nllnnn}
Let $n\ge 3$ and further let $(I,\Gamma)$ and $(J,\Delta)$ be two
form ideals of $(A,\Lambda)$. Then
$$ \EU(2n,IJ+JI,{}^J\Gamma+{}^I\Delta+\Gamma_{\min}(IJ+JI))\le
\FU(2n,I+J,\Gamma+\Delta). $$
\end{Lem}
\begin{proof} In~\cite[Theorem 2.3]{RH} this Lemma is proved the case that $IJ=JI$.
The similar proof shows that elements of the form
$$ \EU(2n,IJ,{}^J\Gamma+{}^I\Delta+\Gamma_{\min}(IJ)),\qquad
\EU(2n,JI,{}^J\Gamma+{}^I\Delta+\Gamma_{\min}(JI)), $$
\noindent
are contained in $\FU(2n,I+J,\Gamma+\Delta)$. By the previous lemma,
the group on the left hand side is their product.
\end{proof}

In the next lemma we calculate the {\it lower\/} level of the mixed
commutator subgroup.
\begin{Lem}\label{nlln}
Let\/ $n\ge 3$. Then for any two for ideals $(I,\Gamma)$ and\/ $(J,\Delta)$
of the form ring\/ $(A,\Lambda)$ one has the following inclusions
$$ \EU(2n,IJ+JI,{}^J\Gamma+{}^I\Delta+\Gamma_{\min}(IJ+JI)) \le
[\EU(2n,I,\Gamma),\EU(2n,J,\Delta)]. $$
\end{Lem}
\begin{proof}
Let $i\neq j$. Take an arbitrary index $h\neq\pm i,\pm j$.
Then the right hand side contains all elementary transvections of
the form
$$ T_{ij}(\a\b)=[T_{ih}(\a),T_{hj}(\b)]\qquad \text{and}\qquad
T_{ij}(\b\a)=[T_{ih}(\b),T_{hj}(\a)], $$
\noindent
for all $\a\in I$, $\b\in J$.
\par
Moreover, being the mutual commutator of two normal subgroups
it is normal in the absolute elementary group $\EU(n,A,\Lambda)$.
Thus, a similar argument as in Lemma~\ref{nllnnn} shows that
$$ \EU(n,IJ+JI,\Gamma_{\min}(IJ+JI))\le [\EU(2n,I,\Gamma),\EU(2n,J,\Delta)]. $$

\par
Furthermore, the right hand side contains
$$ T_{i,-i}(\lambda^{(\epsilon(j)-\epsilon(i))/2}\b\a\b)=
T_{i,-j}(-\b\a)[T_{i,j}(\b),T_{j,-j}(\a)] $$
\noindent
and
$$ T_{i,-i}(\lambda^{(\epsilon(j)-\epsilon(i))/2}\a\b\a)=
T_{i,-j}(-\a\b)[T_{i,j}(\a),T_{j,-j}(\b)]. $$
\noindent
It immediately follows that
$$ \EU(2n,IJ+JI,{}^J\Gamma+{}^I\Delta+\Gamma_{\min}(IJ+JI))\le
[\EU(2n,I,\Gamma),\EU(2n,J,\Delta)]. $$
\end{proof}

Observe, that the above calculation crucially depended on the fact
that $n\ge 3$ and we do not know how to estimate the lower level
for $n=2$ without some strong additional assumptions on the ring $A$.
In the following lemmas we estimate the {\it upper\/} level.
\begin{Lem}\label{GUGU}
Let\/ $n\ge 2$. Then for any two form ideals $(I,\Gamma)$ and\/ $(J,\Delta)$
of the form ring\/ $(A,\Lambda)$ one has the following inclusion
$$ [\GU(2n,I,\Gamma),\GU(2n,J,\Delta)]\le\GU(2n,IJ+JI,\Gamma_{\max}(IJ+JI)). $$
\end{Lem}
\begin{proof}
Take arbitrary $x\in\GU(2n,I,\Gamma)$ and $y\in\GU(2n,J,\Delta)$.
Then $x=e+x_{1}$, $x^{-1}=e+x_{2}$ for some
$x_{1},x_{2}\in M(2n,I)$ such that $x_{1}+x_{2}+x_{1}x_{2}=0$
and $y=e+y_{1}$, $y^{-1}=e+y_{2}$ for some $y_{1},y_{2}\in M(2n,J)$
such that $y_{1}+y_{2}+y_{1}y_{2}=0$. Modulo $IJ+JI$ one has
$$ [x,y]=(e+x_{1})(e+y_{1})(e+x_{2})(e+y_{2})\equiv
e+x_{1}+x_{2}+x_{1}x_{2}+y_{1}+y_{2}+y_{1}y_{2}=e. $$
\noindent
This shows that $[x,y]\in\GL(2n,A,IJ+JI)$. Clearly, $x\in\GU(2n,I,\Gamma)$
and $y\in\GU(2n,J,\Delta)$ preserve the sesquilinear form $f$ 
modulo $\Gamma$ and $\Delta$, respectively, see \S\ref{relative}.
Now, an easy calculation shows that $[x,y]$ preserves $f$ modulo 
$\Gamma+\Delta$. On the other hand, since $x\in GL(2n,A,I)$ it
follows that $f(xu,xu)-f(u,u)\in I$. Putting these observations 
together, we see that $[x,y]$ preserves $f$ modulo 
$$ (IJ+JI)\cap(\Gamma+\Delta)\subseteq (IJ+JI)\cap\Lambda=\Gamma_{\max}(IJ+JI). $$
\noindent
This finishes the proof.
\end{proof}
\begin{Lem}\label{EUGU}
Let\/ $n\ge 3$. Then for any two form ideals\/ $(I,\Gamma)$ and\/ $(J,\Delta)$
of the form ring\/ $(A,\Lambda)$ one has the following inclusion 
$$ [\EU(2n,I,\Gamma),\GU(2n,J,\Delta)]\le
\GU(2n,IJ+JI,{}^J\Gamma+{}^I\Delta+\Gamma_{\min}(IJ+JI)). $$
\end{Lem}
\begin{proof} By the commutator identities (C1) and (C2)
and Lemma~5, it suffices to verify that
$$ g=[T_{lk}(\a),h]\in\GU(2n,IJ+JI,{}^J\Gamma+{}^I\Delta+\Gamma_{\min}(IJ+JI)), $$
\noindent
where $h=(h_{i,j})\in\GU(2n,J,\Delta)$ and $\a\in I$ for $l\ne-k$,
and $\a\in \lambda^{-(\e(i)+1)/2}\Gamma$ for $l=-k$. 
\par
By the previous lemma, we already have a similar inclusion with the
maximal value of relative form parameter. Thus, it only remains to
verify that 
$$ \displaystyle\sum_{1\le i\le n}\bar g_{ij}g_{-i,j}\in
{}^J\Gamma+{}^I\Delta+\Gamma_{\min}(IJ+JI). $$
\par
The proof is divided into two cases depending on whether the root
element $T_{lk}(\alpha)$ is of long or of short type, respectively. 
We attach a detailed calculation for the case of a long root type element. 
The case of a short root type element is settled by a similar calculation
which will be omitted.
\par\smallskip
Let $T_{l,-l}(\a)$ be a long root element, where $\a\in\lambda^{-(\e(l)+1)/2}\Gamma$.
In this case
\begin{eqnarray*}
g=[T_{l,-l}(\a),h]=T_{l,-l}(\a)\Bigl(e-\sum_{i,j}h_{i,l}\a\bar h_{-j,l}\Bigr).
\end{eqnarray*}

\par
Let us have a closer look at the sum $\sum_{1\le i\le n}\bar g_{ij}g_{-i,j}$.
When $j\ne-l$, we may, without loss of generality, assume that $l\ge 0$
and $j\ge 0$, and thus this sum can be rewritten in the form
\begin{multline*}
\sum_{1\le i\le n}\overline{ h_{i,l}\a\bar h_{-j,l}}h_{-i,l}\a\bar h_{-j,l}
-\lambda^{(\e(j)-\e(-l))/2} h_{-j,l}\a\bar h_{-j,l}+
\overline { \a h_{-l,l}\a\bar h_{-j,l}}h_{-l,l}\a\bar h_{-j,l}  \\
=\sum_{1\le i\le n}h_{-j,l}\lambda\bar\a
\bar h_{i,l} h_{-i,l}\a\bar h_{-j,l} -
h_{-j,l}\bar\lambda\a\bar h_{-j,l}+
h_{-j,l}\bar\a \bar h_{-l,l}\bar\a h_{-l,l}\a\bar h_{-j,l},
\end{multline*}
\noindent
where the first summand belongs to ${}^I\Delta$, whereas the second and
the third ones belong to ${}^J\Gamma$, as claimed.
\par
On the other hand, when $j=-l$, this sum equals
$$ \sum_{1\le i\le n}\overline{h_{il}\a\bar h_{ll}}h_{-i,l}\a\bar h_{ll}
-h_{ll}\bar\a\bar h_{ll}+\big(\bar\a-\overline{\a h_{-l,l}\a\bar h_{ll}})
(1-h_{-l,l}\a\bar h_{ll}\big), $$
\noindent
where the first sum belongs to ${}^I\Delta$, while the rest equals
\begin{multline*}
x=-h_{ll}\bar\a\bar h_{ll}+(\bar\a-{h_{ll}\bar\a\bar
h_{-l,l}\bar\a})(1-h_{-l,l}\a\bar h_{ll})= \\
-h_{ll}\bar\a\bar h_{ll}+\bar\a-\bar\a h_{-l,l}\a\bar h_{ll}
-{h_{l,l}\bar\a\bar h_{-l,l}\bar\a}+{ h_{l,l}\bar\a\bar h_{-l,l}\bar\a}
h_{-l,l}\a\bar h_{ll}=\\
-(1+h_{ll}-1)\bar\a(1+\overline{h_{ll}-1})+\bar\a +
\Big(\lambda\bar\a h_{-l,l}\bar\a\bar h_{ll}-
{h_{ll}\bar\a\bar h_{-l,l}\bar\a}\Big)+
h_{ll}\bar\a\bar h_{-l,l}\bar\a h_{-l,l}\a\bar h_{ll} 
\end{multline*}
where the two last summands belong to $\Gamma_{\min}(IJ+JI)$ and to
${}^J\Gamma$, respectively. Thus, modulo ${}^J\Gamma+\Gamma_{\min}(IJ+JI)$
one has
$$ x=-(h_{ll}-1)\bar\a+\lambda\a\overline{(h_{ll}-1)}-
(h_{l,l}-1)\a\overline{(h_{l,l}-1)}, $$ 
\noindent
where the first summand also belongs to $\Gamma_{\min}(IJ+JI)$,
whereas the second one belongs to ${}^J\Gamma$, respectively. 
\par
Thus, in both cases the desired sum belongs to 
${}^J\Gamma+{}^I\Delta+\Gamma_{\min}(IJ+JI)$, as claimed.

\end{proof}




\section{Relative versus absolute, and variations}

Now we are in a position to give another proof of Theorem~\ref{main}.
\begin{proof}[Second proof of Theorem $1$]
By Lemma~\ref{hww3} one has
$$ [\EU(2n,I,\Gamma),\GU(2n,J,\Delta)]=
\big [[\EU(2n,A,\Lambda),\EU(2n,I,\Gamma)],\GU(2n,J,\Delta)\big]. $$
\noindent
Since $(A,\Lambda)$ is a quasi-finite form ring and $n\geq 3$, by Lemma~\ref{keylem}, all the subgroups above are normal in $\GU(2n,A,\Lambda)$. Now  Lemma~\ref{HW1}
implies that 
\begin{multline*}
[\EU(2n,I,\Gamma),\GU(2n,J,\Delta)]\le \\
 \le \big [\EU(2n,I,\Gamma),[\EU(2n,A,\Lambda),\GU(2n,J,\Delta)]\big ]\cdot\\
\big [\EU(2n,A,\Lambda),[\EU(2n,I,\Gamma),\GU(2n,J,\Delta)]\big ].
\end{multline*}
\noindent
Applying to the first factor on the right hand side the
{\it absolute\/} standard commutator formula we immediately see
that it {\it coincides\/} with $[\EU(2n,I,\Gamma),\EU(2n,J,\Delta)]$.
\par
On the other hand, applying to the second factor on the right hand
Lemma~\ref{EUGU} followed by Lemma~\ref{nlln}, we can conclude that it is 
{\it contained\/} in
\begin{multline*}
[\EU(2n,A,\Lambda),\GU(2n,IJ+JI,{}^J\Gamma+{}^I\Delta+\Gamma_{\min}(IJ+JI))]=\\
\EU(2n,IJ+JI,{}^J\Gamma+{}^I\Delta+\Gamma_{\min}(IJ+JI))\le\\
[\EU(2n,I,\Gamma),\EU(2n,J,\Delta)].
\end{multline*}
\noindent
Thus, the left hand side is contained in the right hand side, the
inverse inclusion is obvious.
\end{proof}

It turns out, that for commutative form rings one can prove a
slightly stronger result.
\begin{The}\label{the2}
Let\/ $n\ge 3$, and $(R,\Lambda)$ be a commutative form
ring. Then for any two form ideals\/ $(I,\Gamma)$ and $(J,\Delta)$
of the form ring $(R,\Lambda)$ one has
$$ [\EU(2n,I,\Gamma),\CU(2n,J,\Delta)]=[\EU(2n,I,\Gamma),\EU(2n,J,\Delta)]. $$
\end{The}
The proof of Theorem~\ref{the2} repeats this proof word for word, but the
reference to Lemma~\ref{EUGU} should be replaced by the reference to the
following slightly stronger Lemma.
\begin{Lem}\label{EUCU}
Let\/ $n\ge 3$ and $(R,\Lambda)$ be a commutative
form ring. Then for any two form ideals $(I,\Gamma)$ and\/ $(J,\Delta)$
of the form ring\/ $(R,\Lambda)$ one has the following inclusion
$$ [\EU(2n,I,\Gamma),\CU(2n,J,\Delta)]\le
\GU(2n,IJ,{}^J\Gamma+{}^I\Delta+\Gamma_{\min}(IJ)). $$
\end{Lem}
This lemma is verified by calculations closely imitating those used 
to establish Lemmas \ref{GUGU} and \ref{EUGU}. However, the difference 
is that now the element $y$ figuring in the proof of Lemma~\ref{GUGU} 
is congruent modulo $J$ not to $e$ itself, but to some $\beta e$, 
where $\beta$ is a unit of the ring $R/J$. It remains to observe
that when $\beta$ is {\it central\/} in $R$, the argument goes 
through without any changes.
\par 
One can show by examples that Lemma~\ref{EUCU} definitely fails for 
non-commutative rings. The reason is as follows. By the very definition 
of $\CU(2n,J,\Delta)$, the above element $\beta$ is central modulo 
$J$. However, it does not have to be central in the ring $R$ itself,
and the summands in the proof of Lemma~\ref{GUGU} do not cancel. As
a result, the level may be much higher than expected.
\par
Lemma~\ref{hww3} asserts that the commutator of two elementary subgroups, one
of which is absolute, is itself an elementary subgroup. One can ask,
whether one always has
$$ [\EU(2n,I,\Gamma),\EU(2n,J,\Delta)]=
\EU(2n,IJ+JI,{}^J\Gamma+{}^I\Delta+\Gamma_{\min}(IJ+JI)). $$
\noindent
Easy examples show that in general this equality may fail quite
spectacularly. In fact, when $I=J$, one can only conclude that
$$ \EU(2n,I^2,\Gamma^2)\le [\EU(2n,I,\Gamma),\EU(2n,I,\Gamma)]\le\EU(2n,I,\Gamma). $$
\noindent
with right bound attained for some {\it proper\/} ideals, such as an
ideal $A$ generated by a central idempotent.
\par
Nevertheless, the true reason, why the equality in Lemma~\ref{hww3} holds, is
not the fact that one of the ideals $I$ or $J$ coincides with $A$,
but only the fact that $I$ and $J$ are comaximal.
\begin{The}
Let\/ $n\ge 3$, and\/ $(A,\Lambda)$ be an arbitrary form
ring for which absolute standard commutator formulae are satisfied.
Then for any two comaximal form ideals\/ $(I,\Gamma)$ and $(J,\Delta)$
of the form ring $(A,\Lambda)$, $I+J=A$, one has the following equality
$$ [\EU(2n,I,\Gamma),\EU(2n,J,\Delta)]=
\EU(2n,IJ+JI,{}^J\Gamma+{}^I\Delta+\Gamma_{\min}(IJ+JI)). $$
\end{The}
\begin{proof}
First of all, observe that by Lemmas 3 and \ref{nllnnn} one has
\begin{multline*}
\EU(2n,I,\Gamma)=[\EU(2n,I,\Gamma),\EU(n,A,\Lambda)]=\\
[\EU(2n,I,\Gamma),\EU(2n,I,\Gamma)\cdot \EU(2n,J,\Delta)].
\end{multline*}
\noindent
Thus,
\begin{multline*}
\EU(2n,I,\Gamma)\le [\EU(2n,I,\Gamma),\EU(2n,I,\Gamma)]\cdot
[\EU(2n,I,\Gamma),\EU(2n,J,\Delta)]\le\\
\le [\EU(2n,I,\Gamma),\EU(2n,I,\Gamma)]\cdot
\GU(2n,IJ+JI,{}^J\Gamma+{}^I\Delta+\Gamma_{\min}(IJ+JI)).
\end{multline*}
\noindent
Commuting this inclusion with $\EU(2n,J,\Delta)$, we see that
\begin{multline*}
[\EU(2n,I,\Gamma),\EU(2n,J,\Delta)]\le
\big [[\EU(2n,I,\Gamma),\EU(2n,I,\Gamma)],\EU(2n,J,\Delta)\big ]\cdot\\
[\GU(2n,IJ+JI,{}^J\Gamma+{}^I\Delta+\Gamma_{\min}(IJ+JI)),\EU(2n,J,\Delta)].
\end{multline*}
\par
The absolute standard commutator formula, applied to the second
factor, shows that its is contained in
\begin{multline*}
[\GU(2n,IJ+JI,{}^J\Gamma+{}^I\Delta+\Gamma_{\min}(IJ+JI)),\EU(2n,J,\Delta)]\le\\
[\GU(2n,IJ+JI,{}^J\Gamma+{}^I\Delta+\Gamma_{\min}(IJ+JI)),\EU(n,A,\Lambda)]=\\
\EU(2n,IJ+JI,{}^J\Gamma+{}^I\Delta+\Gamma_{\min}(IJ+JI)).
\end{multline*}
\par
On the other hand, applying to the first factor Lemma \ref{EUGU}, 
and then again the absolute standard commutator formula, we see 
that it is contained in
\begin{multline*}
[[\EU(2n,I,\Gamma),\EU(2n,J,\Delta)],\EU(2n,I,\Gamma)]\le \\
[\GU(2n,IJ+JI,{}^J\Gamma+{}^I\Delta+\Gamma_{\min}(IJ+JI)),\EU(2n,I,\Gamma)]\le\\
\le[\GU(2n,IJ+JI,{}^J\Gamma+{}^I\Delta+\Gamma_{\min}(IJ+JI)),\EU(2n,A,\Lambda)]=\\
\EU(2n,IJ+JI,{}^J\Gamma+{}^I\Delta+\Gamma_{\min}(IJ+JI)).
\end{multline*}
\par
Together with Lemma \ref{nllnnn} this finishes the proof.
\end{proof}


\section{Where next?}

In this section we state and very briefly discuss some further
relativisation problems, related to the results of the present paper.
We are convince that these problems can be successfully addressed
with our methods.
\par
In the following problems we propose to generalise results by
Sivatsky--Stepanov \cite{SiSt} and Stepanov--Vavilov \cite{SV10}
to Bak's unitary groups.
\begin{prob}
\label{p1} Obtain explicit length estimates in the relative
conjugation calculus and commutator calculus.
\end{prob}
\begin{prob}
\label{p2} Let $\jdim(R)<\infty$. Prove that the width of commutators
in elementary generators is bounded, and estimate this width.
\end{prob}
Alexei Stepanov (unpublished) established that the above width is bounded,
without actually producing any specific bound. We believe that the methods
of the present paper allow to give an {\it exponential\/} bound, similar
to the one obtained for Chevalley groups over commutative rings \cite{SV10},
by developing a constructive version of the localisation method
from Hazrat--Vavilov \cite{RN1}. We believe that obtaining a similar
constructive version of the results of the present paper would be simply
a matter of patience. On the other hand, to obtain a {\it polynomial\/}
bound, similar to that obtained for $\GL(n,A)$ in \cite{SiSt}, one would
need to combine our methods with a {\it full-scale\/} generalisation of
decomposition of unipotents \cite{SV}, including the explicit polynomial
formulae for the conjugates of root unipotents. This seems to be somewhat
remote.
\par
In the main results of the present paper we always assume that $n\ge 3$.
Obviously, due to the exceptional behaviour of the orthogonal group
$\SO(4,A)$, these results do not fully generalise to the case, where
$n=2$. However, we believe they do generalise under appropriate additional
assumptions on the form ring, such as $\Lambda A+A\Lambda=\Lambda$.
Known results, including the work by Vyacheslav Kopeiko \cite{kopeiko}
and the work by Bak--Vavilov \cite{BV2} clearly indicate both that
this should be possible, and that the analysis of the case $n=2$ be
considerably harder from a technical viewpoint, than that of the case
$n\ge 3$.
\begin{prob}
\label{p3} Develop conjugation calculus and commutator calculus for
in the group $\GU(4,A,\Lambda)$, provided $\Lambda A+A\Lambda=\Lambda$.
\end{prob}
\begin{prob}
\label{p4} Prove relative standard commutator formula for the group
$\GU(4,A,\Lambda)$, provided $\Lambda A+A\Lambda=\Lambda$.
\end{prob}
Solution of the following problem would be a broad generalisation of
Bak \cite{B4}, Hazrat \cite{RH,RH2}, and Bak--Hazrat--Vavilov \cite{BRN}.
Clearly, it will require the full force of localisation--completion.
\begin{prob}
\label{p5} Let $R$ be a ring of finite Bass--Serre dimension
$\delta(R)=d<\infty$, and let $(I_i,\Gamma_i)$, $1\le i\le m$, be
form ideals of $(A,\Lambda)$. Prove that for any $m>d$ one has
\begin{multline*}
[[\ldots[\GU(2n,I_1,\Gamma_1),\GU(2n,I_2,\Gamma_2)],\ldots],\GU(2n,I_m,\Gamma_m)]=\\
[[\ldots[\EU(2n,I_1,\Gamma_1),\EU(2n,I_2,\Gamma_2)],\ldots],\EU(2n,I_m,\Gamma_m)].
\end{multline*}
\end{prob}

Let us mention also generalisation of the results of the present paper
to other types of groups. In view of \cite{RN1,BRN,SV10} the first of
the problems below seems almost immediate, and it is our intention to
address it in a subsequent paper.
\begin{prob}
\label{p6} Obtain results similar to those of the present paper
for Chevalley groups.
\end{prob}
The other two problems, especially the last one, seem to be {\it much\/}
more challenging, from a technical viewpoint. In both cases root subgroups
are not abelian, and the analogues of the Chevalley commutator formula
are much fancier, than in the familiar cases of Chevalley groups,
or Bak's unitary groups. As a matter of fact, the required version of
localisation has not been developed in either of these contexts, even at
the absolute level.
\par
The following problem refers to the context of odd unitary groups,
as created by Victor Petrov \cite{P1,petrov2,petrov3}.
\begin{prob}
\label{p7} Generalise results of the present paper to odd unitary groups.
\end{prob}
The last problem refers to the recent context of isotropic reductive
groups. Of course, it only makes sense over commutative rings, but on the
other hand, a lot of new complications occur, due to the fact that
relative roots do not form a root system, and the interrelations of
the elementary subgroup with the group itself are abstruse even over
fields (the Kneser--Tits problem). Still, we are convinced that most
of necessary tools are already there, in the remarkable recent papers
by Victor Petrov and Anastasia Stavrova, \cite{PS08,stavrova}. Of
course, one will have to develop the whole conjugation and commutator
calculus almost from scratch.
\begin{prob}
\label{p8} Obtain results similar to those of the present paper for
[groups of points of] isotropic reductive groups.
\end{prob}

The authors thank Alexei Stepanov for extremely useful discussions.


\end{document}